\documentclass[11pt]{amsart}%
\usepackage{amscd,amsmath,latexsym,amsthm,amsfonts,amssymb,graphicx,color,geometry}
\usepackage{enumerate}
\usepackage[pdftex,plainpages=false,pdfpagelabels,
colorlinks=true,citecolor=blue,linkcolor=black,urlcolor=black,
filecolor=black,bookmarksopen=true,unicode=true]{hyperref}
\usepackage[norefs,nocites]{refcheck}
\usepackage{amsmath}
\usepackage{amsfonts}
\usepackage{amssymb}
\usepackage[T1]{fontenc}
\usepackage{graphicx}
\usepackage{xcolor}
\usepackage{colortbl}
\usepackage{float}
\usepackage{framed}
\usepackage{pgf,tikz,pgfplots}
\usepackage{tkz-base}
\usepackage{multicol}
\usepackage{tabularx}
\usepackage{tkz-fct}%
\providecommand{\U}[1]{\protect\rule{.1in}{.1in}}
\tikzset{>=Triangle}
\newtheorem{theorem}{Theorem}[section]
\newtheorem{proposition}[theorem]{Proposition}

\newtheorem{corollary}[theorem]{Corollary}

\newtheorem{lemma}[theorem]{Lemma}

\newtheorem{definition}[theorem]{Definition}
\numberwithin{equation}{section}
\pgfplotsset{compat=1.17}
\begin{document}
\title[On Sequences with at Most a Finite Number of Zero Coordinates]{On Sequences with at Most a Finite Number of Zero Coordinates}
\author[Diego]{Diego Alves}
\address{Departamento de Ensino \\
Instituto Federal do Ceará \\
 63708-260 - Crateús, Brazil.}
\email{diego.costa@ifce.edu.br}
\author[Geivison]{Geivison Ribeiro}
\address{Departamento de Matem\'{a}tica \\
Universidade Federal da Para\'{\i}ba \\
58.051-900 - Jo\~{a}o Pessoa, Brazil.}
\email{geivison.ribeiro@academico.ufpb.br}
\thanks{G. Ribeiro is supported by Grant 2022/1962, Para\'{\i}ba State Research
Foundation (FAPESQ)}

\begin{abstract}
In this paper, we analyze the existence of algebraic and topological structures present in the set of sequences that contain only a finite number of zero coordinates. Inspired by the work of Daniel Cariello and Juan B. Seoane-Sepúlveda, published in [J. Funct. Anal. \textbf{266} (2014), 3797--381], our research reveals new insights and complements their notable results beyond the $\ell_{p}$ spaces for $p \in [1, \infty]$, including the intriguing case where $p \in (0, 1)$.

Our exploration employs notions such as $[\mathcal{S}]$-lineability, pointwise lineability, and $(\alpha, \beta)$-spaceability. This investigation allowed us to verify, for instance, that the set $F \setminus Z(F)$ (where $F$ is a closed subspace of $\ell_{p}$ containing $c_{0}$) is $(\alpha, \mathfrak{c})$-spaceable if and only if $\alpha$ is finite.
\end{abstract}
\maketitle

\section{Introduction and background}

The investigation of the algebraic and topological structures within the set of sequences with only a finite number of zero was initiated by R.M. Aron and V.I. Gurariy in 2003 during a Non-linear Analysis Seminar at Kent State University (Kent, Ohio, USA). They posed the following question:
	
	\textbf{Question (R. Aron \& V. Gurariy, 2003):}  
	\textit{Is there an infinite dimensional closed subspace of the space of bounded sequences where every nonzero element of which has only a finite number of zero coordinates?}
	
	In 2013, Cariello and Seoane-Sepúlveda provided an answer to this question (see \cite{Cariello}). Among their results, they proved that the subset \( Z(X) \) of \( X \) (where \( X = c_0 \) or \( \ell_p \) with \( p \in [1, \infty] \)) formed by sequences with only a finite number of zero coordinates does not contain an infinite dimensional closed subspace, except for the zero vector.
	
 Building on this work, our paper proceeds with such investigation, albeit now employing interconnected notions and techniques. To provide context, let us fix some notations and terminology. We denote \( \mathbb{N} \) as the set of positive integers, \( \mathbb{R} \) as the real scalar field, and $\mathbb{C}$ as the complex scalar field. Unless explicitly stated otherwise, all linear spaces are over \(\mathbb{K}= \mathbb{R} \text{ or } \mathbb{C} \). The term subspace will be used instead of vector subspace. Additionally, \( \alpha \) and \( \beta \) represent cardinal numbers, \( \text{card}(A) \) denotes the cardinality of set \( A \), \( \aleph_0 := \text{card}(\mathbb{N}) \), and \( \mathfrak{c} := \text{card}(\mathbb{R}) \). For \( p \in (0, \infty) \), as usual, we denote \( \ell_p \) as the classical set of absolutely \( p \)-summable sequences, defined as: \[
\ell_{p}:=\left\{  \left(  t_{n}\right)  _{n=1}^{\infty}\in\mathbb{K}%
^{\mathbb{N}}:%
{\textstyle\sum\limits_{n=1}^{\infty}}
\left\vert t_{n}\right\vert ^{p}<\infty\right\}  \text{.}%
\]
Recall that if $p\in(0,1]$, then a \textit{$p$-Banach space} $X$ is a vector space
equipped with a non-negative function denoted by $\left\Vert \cdot\right\Vert_p
$, also known as a $p$-norm, which satisfies the following conditions:

\begin{enumerate}
\item $\left\Vert x\right\Vert _{p}=0$ if and only if $x=0$.

\item $\left\Vert tx\right\Vert _{p}=\left\vert t\right\vert ^{p}\left\Vert
x\right\Vert _{p}$ for all $x\in X$ and $t\in\mathbb{K}$.

\item $\left\Vert x+y\right\Vert _{p}\leq\left\Vert x\right\Vert
_{p}+\left\Vert y\right\Vert _{p}$ for all $x,y\in X$.
\end{enumerate}
It is known that for \( p \in (0,1] \), \( \ell_{p} \) is a \( p \)-Banach space with the \( p \)-norm given by \( \lVert (t_{n})_{n=1}^{\infty} \rVert_{p} := \sum_{n=1}^{\infty} |t_{n}|^{p} \). Although not much is known about the structure of these spaces, they still play an important role in the geometry of $F$-spaces (metrizable topological vector spaces). For more information on the theory of \( p \)-Banach spaces, we refer to \cite{Kalt, Stiles}.

Still in terms of notations and terminology, as usual, we denote $\ell_{\infty}$ as the classical space of
bounded sequences defined as:
\[
\ell_{\infty}:=\left\{  \left(  t_{n}\right)  _{n=1}^{\infty}\in
\mathbb{K}^{\mathbb{N}}:\underset{n\in\mathbb{N}}{\sup}\left\vert
t_{n}\right\vert <\infty\right\}  \text{.}%
\]
For a vector space \( X \), we say that a set \( A \subseteq X \) is \( \alpha \)-lineable if it contains an \( \alpha \)-dimensional subspace of \( X \), excluding the zero vector. Moreover, if \( X \) has a topology, the subset \( A \) is \( \alpha \)-spaceable in \( X \) whenever it contains a closed \( \alpha \)-dimensional subspace of \( X \), excluding the zero vector. For the case where $\alpha$ is infinite, for simplicity, we say that $A$ is \textit{lineable} (respectively \textit{spaceable}). This concepts was first introduced in the seminal paper \cite{AGSS} by Aron, Gurariy and Seoane-Sep\'{u}lveda and, later, in \cite{Enflo, Gurariy-Quarta}. Its essence is to investigate linear structures within exotic settings.
 
A more detailed description of the concept of lineability/spaceability can also be found in \cite{Aron, Bernal, Botelho, Cariello, Enflo, Fonf}. In the quest for a deeper understanding of dimensional relationships within spaceable sets, Fávaro, Pellegrino, and Tomaz introduced a more stringent concept called $\left(  \alpha,\beta\right)  $\textit{-spaceability} in 2019 (referenced as \cite{FPT}). This refined definition posits that a set \(A\) in a topological vector space \(X\) is considered $\left(  \alpha,\beta\right)  $\textit{-spaceable} if it is $\alpha$-lineable, and for each $\alpha$-dimensional subspace \(V_{\alpha}\) contained in \(A \cup \{ 0 \}\), there exists a closed $\beta$-dimensional subspace \(F_{\beta}\) satisfying \(V_{\alpha} \subseteq F_{\beta} \subseteq A \cup \{ 0 \}\).

This heightened conceptualization, although rooted in conventional spaceability, demands meticulous consideration, encompassing both geometric and topological properties within the set, as well as relationships between the dimensions of subspaces therein. It is worth noting that while conventional spaceability implies the variant known as $\left(  \alpha,\alpha\right)  $-spaceability, the reverse assertion is not always valid, as demonstred by Araújo et al$.$ in \cite{Araujo}$.$ For a more comprehensive exposition on this notion, we refer the reader to \cite{Mikaela, Nacib, BRSH, Sheldon, Diogo/Anselmo,
Diogo, Pilar, FPT, Pellegrino}.

Connected to the notion of lineability, for a subspace \( \mathcal{S} \) of \( \mathbb{K}^{\mathbb{N}} \), we say that a subset \( A \) of a Hausdorff topological vector space \( X \) is \( [ (u_{n})_{n=1}^{\infty}, \mathcal{S} ] \)-lineable in \( X \) if, for each sequence \( (c_{n})_{n=1}^{\infty} \in \mathcal{S} \), the series \( \sum_{n=1}^{\infty} c_{n} u_{n} \) converges in \( X \) to a vector in \( A \cup \{ 0 \} \). Moreover, \( A \) is \( [ \mathcal{S} ] \)-lineable in \( X \) if it is \( [ (u_{n})_{n=1}^{\infty}, \mathcal{S} ] \)-lineable for some sequence \( (u_{n})_{n=1}^{\infty} \) of linearly independent elements in \( X \). 

This concept was introduced in \cite{S lineability} by Bernal-González et al$.$ and originally coined by V. Gurariy and R. Aron during a Non-Linear Analysis Seminar at Kent State University. As far as we know, this notion was initially inspired by a result of Levine and Milman from 1940 \cite{Levine e Milman}, created to address the lack of such a concept regarding convergence. Regarding the result of Levine and Milman, they demonstrate that the subset of \( C[0,1] \) consisting of all functions of bounded variation does not contain a closed subspace of infinite dimension.

Other concepts related to the notion of lineability will be introduced throughout the subsequent sections.
\bibliographystyle{plain}

This paper is organized as follows. In Section \ref{section 2}, we focus on
the non spaceability of the set $Z(X)$ for $X\in\{c_{0},\ell_{p}\}$,
$p\in\lbrack1,\infty]$, as demonstrated by the authors in \cite{Cariello}, and
establish that, despite this, the set $Z(X)$ for $X\in\{c_{0},\ell_{p}\}$,
$p\in(0,\infty)$, is $[\mathcal{S}]$-lineable for every subspace $\mathcal{S}$ of
$\ell_{\infty}$. In Section \ref{Section3}, we prove that $Z(X)$ also enjoys
pointwise lineability. In Section \ref{Section4}, we demonstrate that
the complement of $Z(F)$, where $F$ is an infinite dimensional closed subspace of $X\in\{c_{0}%
,\ell_{p}\}$, $p\in(0,\infty]$, is indeed more than spaceable. We establish
that it is, in fact, $(n,\mathfrak{c})$-spaceable for every positive integer $n$.  Finally, in Section \ref{Section5}, we verify in particular that the set $F \setminus Z(F)$ (where $F$ is a closed subspace of $\ell_{p}$ containing $c_{0}$) is $(\alpha, \mathfrak{c})$-spaceable if and only if $\alpha$ is finite.

\section{A Search for Convergence: $\left[  \mathcal{S}\right]
$-lineability\label{section 2}}

In this section, we explore the convergence of series in $Z\left(  F\right)
\cup\left\{  0\right\}  $, where $F$ is any closed subspace of $\ell_{p}$ (with
$p\in(0,\infty]$) that contains $c_{0}$. We begin with a fundamental lemma on the
passage of limits in series, which will allow us to achieve our objective.

\begin{lemma}\label{o lema que permite passar o limite pra dentro}
 Let  $r$ be a natural number and
$\left( t_{n}\right)  _{n=r+1}^{\infty}$ be a sequence in $\ell_{\infty}.$ Suppose that $\left(
k_{j}\right)_{j=1}^{\infty}$ is an unbounded increasing sequence of real numbers with $k_{1}>2.$ Then 
\[
\underset{j\rightarrow\infty}{\lim}\sum\limits_{n=r+1}^{\infty}t_{n}\left(
\frac{n}{r}\right)  ^{-k_{j}}=0\text{.}%
\]
\end{lemma}
\begin{proof}
Consider the set $\mathbb{N}_{r}=\left\{  n\in\mathbb{N}:n>r\right\}  $. For each $j\in
\mathbb{N}$, define $f_{j}\colon\mathbb{N}_{r} \longrightarrow \mathbb{K}$ as
\[
f_{j}\left(  n\right)  :=t_{n}\left(  \frac{n}{r}\right)  ^{-k_{j}}= t_n\left(
\frac{r}{n}\right)^{k_{j}}\text{.}%
\]
Note that for any $j \in \mathbb{N}$
\[
\left(  \frac{n}{r}\right)^{-k_{j}}<\left(  \frac{n}{r}\right)^{-2 } \Longleftrightarrow \left(  \frac{r}{n}\right)  ^{k_{j}}<\left(  \frac{r}{n}\right)^{2} \Longleftrightarrow r^{k_{j}-2}<n^{k_{j}-2}\Longleftrightarrow 
r<n\text{.}
\]
Thus, for every $n\in\mathbb{N}_{r}$, we get%
\[
\left\vert f_{j}\left(n\right) \right\vert =\left\vert t_{n}\right\vert
\left(\frac{n}{r}\right)^{-k_{j}} \leq C \left(  \frac{n}{r}\right)
^{-k_{j}} \leq C\left(  \frac{n}{r}\right)^{-2},
\]
where $C= \left\Vert \left(  t_{n}\right)  _{n=r+1}^{\infty}\right\Vert
_{\ell_{\infty}}.$ Therefore, we have
\[
\sum\limits_{n\in\mathbb{N}_{r}}\left\vert f_{j}\left(  n\right) \right\vert
 \leq C \sum\limits_{n\in\mathbb{N}_{r}}\left(  \frac{n}{r}\right) ^{-2}=C r^{2}
\sum\limits_{n\in\mathbb{N}_{r}}\dfrac{1}{n^{2}}<\infty\text{.}
\]
Thus, for any $\varepsilon>0$, there exists $n_{0}\in\mathbb{N}_{r}$ such that%
\[
\sum\limits_{n\geq n_{0}}\left\vert f_{j}\left(  n\right)  \right\vert
<\varepsilon\text{.}%
\]
Since%
\[
\left\vert \sum\limits_{n\in\mathbb{N}_{r}}f_{j}\left(  n\right)  \right\vert
\leq\sum\limits_{n\in\mathbb{N}_{r}}\left\vert f_{j}\left(  n\right)
\right\vert =\sum\limits_{n<n_{0}}\left\vert f_{j}\left(  n\right)
\right\vert +\sum\limits_{n\geq n_{0}}\left\vert f_{j}\left(  n\right)
\right\vert
\]
and%
\[
\underset{j\rightarrow\infty}{\lim}f_{j}\left(  n\right)  =\underset
{j\rightarrow\infty}{\lim}t_{n}\left(  \frac{n}{r}\right)  ^{-k_{j}}=0\text{,}%
\]
for every $n\in\mathbb{N}_{r}$, we obtain
\begin{align*}
\underset{j\rightarrow\infty}{\limsup}\left\vert \sum\limits_{n\in
\mathbb{N}_{r}}f_{j}\left(  n\right)  \right\vert  &  \leq\underset
{j\rightarrow\infty}{\limsup}\sum\limits_{n<n_{0}}\left\vert f_{j}\left(
n\right)  \right\vert +\underset{j\rightarrow\infty}{\limsup}\sum
\limits_{n\geq n_{0}}\left\vert f_{j}\left(  n\right)  \right\vert \\
&  \leq0+\varepsilon\\
&  =\varepsilon\text{.}%
\end{align*}
Since $\varepsilon$ is arbitrary, we have%
\[
\underset{j\rightarrow\infty}{\limsup}\left\vert \sum\limits_{n\in
\mathbb{N}_{r}}f_{j}\left(  n\right)  \right\vert =0\text{.}%
\]
Moreover,
\[
0\leq\underset{j\rightarrow\infty}{\liminf}\left\vert \sum\limits_{n\in
\mathbb{N}_{r}}f_{j}\left(  n\right)  \right\vert \leq\underset{j\rightarrow
\infty}{\limsup}\left\vert \sum\limits_{n\in\mathbb{N}_{r}}f_{j}\left(
n\right)  \right\vert =0\text{.}%
\]
Thus,%
\[
\underset{j\rightarrow\infty}{\lim}\left\vert \sum\limits_{n\in\mathbb{N}_{r}%
}f_{j}\left(  n\right)  \right\vert =0\text{.}%
\]
Consequently,
\[
\underset{j\rightarrow\infty}{\lim}\sum\limits_{n\in\mathbb{N}_{r}}%
f_{j}\left(  n\right)  =0\text{,}%
\]
as desired.
\end{proof}

The lemma above will be crucial in proving the next result.

\begin{proposition}\label{linea}
For $p\in(0,\infty]$, the set $Z\left(  \ell_{p}\right)  $ is $\left[
\ell_{\infty}\right]  $-lineable.
\end{proposition}

\begin{proof}
For each $n>1$ $\left(  n\in\mathbb{N}\right)  $, consider $a_{n}:=\frac{n-1}{n}$,
and define the sequence
$$u_n:=
\left\{
\begin{array}{ccc}
\left(a_n^{1/p}n^{-k/p} \right)_{k=0}^{\infty}, & \text{if} & p \in (0, \infty), \\\\

\left(n^{-k} \right)_{k=0}^{\infty}, & \text{if} & p=  \infty. \\
\end{array}
\right.
$$
Since%
\[
\sum\limits_{k=0}^{\infty}\left\vert a_{n}^{1/p} \, n^{-k/p }\right\vert ^{p}=\sum\limits_{k=0}^{\infty}a_{n}n^{-k}=a_{n}\sum
\limits_{k=0}^{\infty}n^{-k}= a_{n}\frac{n}{n-1}=1
\]
and%
\[
\underset{k\geq 0 }{\sup}\left\vert
n^{-k}\right\vert =1
\]
for each $n>1$, we have that the sequence $\left(  u_{n}\right)  _{n=2}^{\infty}\subset
\ell_{p}$ satisfies $\left\Vert u_{n}\right\Vert _{p}=1$ for each $n>1$. We claim that the sequence $\left(  u_{n}\right)  _{n=2}^{\infty}$ is linearly independent in $\ell_{p}$. To prove this, consider scalars $\alpha_{1},\alpha_{2},\ldots,\alpha_{N}$ such that

\begin{equation}\label{osalfaisnaotodosnulos}
 \alpha_{1} \, a_{n_{1}}^{1/p} \, n_{1}^{-k/p}+\alpha_{2} \, a_{n_{2}}^{1/p} \, n_{2}^{-k/p} + \ldots +\alpha_{N} \,a_{n_{N}}^{1/p} \,
n_{N}^{-k/p}=0, \, \,  
\end{equation}
for all $k \geq 0.$ Without loss of generality, we can assume that $n_{1}<n_{2}<\ldots<n_{N}$. Then, multiplying (\ref{osalfaisnaotodosnulos}) by $n_1^{k/p},$ we get
\[
\alpha_{1} a_{n_{1}}^{1/p}+\alpha_{2}a_{n_{2}
}^{1/p}\left(  \frac{n_{2}}{n_{1}}\right)  ^{-k/p}
+\cdots+\alpha_{N}a_{n_{N}}^{1/p}\left(  \frac{n_{N}}{n_{1}}\right)
^{-k/p}=0,
\]
for all $k \geq 0 $ and for all $i\in\left\{
2,\ldots,N\right\}.$ Since the terms $\left(  \frac{n_{i}}{n_{1}}\right)  ^{-k}$ tend to $0$ as $k$ tends to infinity, we can infer that $\alpha_{1} a_{n_1}^{1/p}=0,$ and consequently, $\alpha_{1}=0.$

According to $\left(  \text{\ref{osalfaisnaotodosnulos}}\right),$ we 
have:
\[
\alpha_{2} \, a_{n_{2}}^{1/p} \, n_{2}^{-k/p}+ \cdots + \alpha_{N} \, a_{n_{N}}^{1/p} \, n_{N}^{-k/p}=0,
\]
for all $k \geq 0 $. Proceeding in the same way, we conclude that $\alpha_2= \cdots = \alpha_N=0.$ This argument concludes that $(u_n)_{n=2}^{\infty}$ is linearly independent in the case where $p \in (0, \infty).$ 
The case $p= \infty$ can be treated in an analogous way. Now, let
\begin{align*}
T\colon\ell_{\infty}  &  \rightarrow\ell_{p}\\
\left(  t_{n}\right)  _{n=2}^{\infty}  &  \mapsto%
{\textstyle\sum\limits_{n=2}^{\infty}}
t_{n}v_{n}\text{,}%
\end{align*}
where $v_{n}:=2^{-n}u_{n}$ for every $n>1$. Let us show that $T$ is well
defined. In fact, we have that
\begin{align*}
\sum\limits_{n=2}^{\infty}\left\Vert t_{n}v_{n}\right\Vert _{p}  &  =\left\{
\begin{array}
[c]{ccc}%
\sum\limits_{n=2}^{\infty}\left\vert t_{n}\right\vert \left\Vert
v_{n}\right\Vert _{p}\text{,} & \text{if} & p\in\left[  1,\infty\right]
\text{,}\\
\sum\limits_{n=2}^{\infty}\left\vert t_{n}\right\vert ^{p}\left\Vert
v_{n}\right\Vert _{p}\text{,} & \text{if} & p\in\left(  0,1\right)  \text{.}%
\end{array}
\right. \\
&  \leq\left\{
\begin{array}
[c]{ccc}%
C\sum\limits_{n=2}^{\infty}\left\Vert v_{n}\right\Vert _{p}\text{,} &
\text{if} & p\in\left[  1,\infty\right]  \text{,}\\
C^{p}\sum\limits_{n=2}^{\infty}\left\Vert v_{n}\right\Vert _{p}\text{,} &
\text{if} & p\in\left(  0,1\right)  \text{.}%
\end{array}
\right. \\
&  =\left\{
\begin{array}
[c]{ccc}%
C\sum\limits_{n=2}^{\infty}2^{-n}\left\Vert u_{n}\right\Vert _{p}\text{,} &
\text{if} & p\in\left[  1,\infty\right]  \text{,}\\
C^{p}\sum\limits_{n=2}^{\infty}2^{-np}\left\Vert u_{n}\right\Vert _{p}\text{,}
& \text{if} & p\in\left(  0,1\right)  \text{.}%
\end{array}
\right. \\
&  =\left\{
\begin{array}
[c]{ccc}%
C\sum\limits_{n=2}^{\infty}2^{-n}\text{,} & \text{if} & p\in\left[
1,\infty\right]  \text{,}\\
C^{p}\sum\limits_{n=2}^{\infty}2^{-np}\text{,} & \text{if} & p\in\left(
0,1\right)  \text{.}%
\end{array}
\right. \\
&  <\infty\text{,}%
\end{align*}
where $C:=\left\Vert \left(  t_{n}\right)  _{n=2}^{\infty}\right\Vert
_{\infty}$. Thus, given $\varepsilon>0$, there exists $n_{0}\in\mathbb{N}$ such that $\sum\limits_{n=n_{0}%
+1}^{\infty}\left\Vert t_{n}v_{n}\right\Vert _{p}<\varepsilon$. Then, for
$r,s\in\mathbb{N}$ with $r>s\geq n_{0}$, we have
\[
\left\Vert \sum\limits_{n=2}^{r}t_{n}v_{n}-\sum\limits_{n=2}^{s}t_{n}%
v_{n}\right\Vert _{p}=\left\Vert \sum\limits_{n=s+1}^{r}t_{n}v_{n}\right\Vert
_{p}\leq\sum\limits_{n=s+1}^{r}\left\Vert t_{n}v_{n}\right\Vert _{p}\leq
\sum\limits_{n=n_{0}+1}^{\infty}\left\Vert t_{n}v_{n}\right\Vert
_{p}<\varepsilon\text{.}%
\]
Since $\ell_{p}$ is an $F$-space $\left(  \textbf{a complete metrizable
topological vector space}\right)  $, we can conclude that sequence $$\left(
\sum\limits_{n=2}^{r}t_{n}v_{n}\right)  _{r=2}^{\infty}$$ converges in
$\ell_{p}$ for the series $\sum\limits_{n=2}^{\infty}t_{n}v_{n}$. That is, the
series $\sum\limits_{n=2}^{\infty}t_{n}v_{n}\in\ell_{p}$ and therefore, $T$ is
well defined. It is plain that $T$ is linear. Due to the fact that%
\[
T\left(  \ell_{\infty}\right)  \subset\ell_{p}\text{,}%
\]
it remains to show that%
\[
T\left(  \ell_{\infty}\right)  \subset Z\left(  \ell_{p}\right)  \cup\left\{
0\right\}  \text{.}%
\]
We will make the case $p \in (0, \infty),$ the case $p= \infty$ is proved in an analogous way. To do this, consider an element $$x=\sum\limits_{n=2}^{\infty}t_{n}u_{n}=\left(  \sum
\limits_{n=2}^{\infty}t_{n}a_{n}^{1/p},\sum\limits_{n=2}^{\infty}
t_{n}a_{n}^{1/p}n^{-1/p},\sum\limits_{n=2}^{\infty}t_{n}
a_{n}^{1/p}n^{-2/p},\sum\limits_{n=2}^{\infty}t_{n}
a_{n}^{1/p}n^{-3/p},\ldots\right)$$ of $T\left(  \ell_{\infty
}\right)  \setminus\left\{  0\right\}.$ Suppose, for
contradiction, that there exists an increasing sequence of positive integers
$\left(  k_{j}\right)  _{j=1}^{\infty}$ with $k_{1}>2p$ such that for each
$j\in\mathbb{N}$,%
\begin{equation}
\sum\limits_{n=2}^{\infty}t_{n} \, a_{n}^{1/p} \, n^{-k_j / p} =0\text{. } \label{a serie infinita de n elevado a menos kj}%
\end{equation}
The plan here is to show that $x=0$, leading to a contradiction. Let us start
by verifying that $t_{2}=0$. Multiplying (\ref{a serie infinita de n elevado a menos kj}) by $2^{k_j/p}$ we obtain
\begin{equation}\label{aa}
t_{2} \, a_{2}^{1/p} \, +\sum\limits_{n=3}^{\infty}t_{n}a_{n}^{1/p}%
\left(  \frac{n}{2}\right)  ^{-k_{j}/p}=0,
\end{equation}
for each $j\in\mathbb{N}$. Since%
\[
\left\vert t_{n}a_{n}^{1/p}\right\vert =\left\vert t_{n}\right\vert
a_{n}^{1/p}\leq\left\vert t_{n}\right\vert \text{ and }\left(
t_{n}\right)  _{n=2}^{\infty}\in\ell_{\infty}\text{,}%
\]
we conclude $(t_n a_n^{1/p})_{n=3}^{\infty} \in \ell_{\infty}.$ Since $(k_j/p)_{j=1}^{\infty}$ is an unlimited increasing sequence of real numbers such that $k_{1}/p>2$ we obtain by Lemma \ref{o lema que permite passar o limite pra dentro} that $$\displaystyle\underset{j\rightarrow\infty}{\lim}\sum\limits_{n=3}^{\infty}%
t_{n}a_{n}^{1/p}\left( \dfrac{n}{2}\right)^{-k_{j}/p}=0.$$
Hence, we have $t_{2}a_{2}^{1/p}=0,$ and consequently, $t_{2}=0.$ Now, let us conclude that $t_3=0.$ From (\ref{aa}) and $t_2=0,$ we have
\[
\sum\limits_{n=3}^{\infty}t_{n} \, a_{n}^{1/p} \, n^{-k_j/p}
=0\text{.}%
\]
multiplying the last identity by $3^{k_{j}/p}$, we obtain
\[
t_{3}a_{3}^{1/p}+\sum\limits_{n=4}^{\infty}t_{n} a_{n}^{1/p}\left( \frac{n}{3}\right)^{-k_{j}/p}=0.
\]
Again applying the Lemma \ref{o lema que permite passar o limite pra dentro} again, we get $t_{3}a_{3}^{1/p}=0,$ and therefore $t_{3}=0.$
Proceeding recursively in this manner, we can conclude that $t_{n}=0$ for each
$n\geq2$. However, this is a contradiction since $x\neq0$. Therefore, the
result is established. 
\end{proof}

\medskip

Given $(t_n)_{n=2}^{\infty} \in \ell_{\infty},$ the arguments employed above show that $t_{n}=0$ for all $n\geq2$, only required the existence of an increasing sequence of
positive integers $\left(  k_{j}\right)_{j=1}^{\infty}$ such that $k_{1}>2$ and
\begin{equation}\nonumber
\sum\limits_{n=2}^{\infty}t_{n} \, a_{n}^{1/p} \, n^{-k_j/p}=0, \, \text{
for each }j\in\mathbb{N}\text{.}%
\end{equation}
Hence, if we assume
\[
T\left(  \left(  t_{n}\right)_{n=2}^{\infty}\right)  =0\text{,}
\]
we will certainly have that $(t_n)_{n=2}^{\infty}$ is the null sequence, which establishes the injectivity of $T.$

\medskip

Moreover, in addition to recovering the result established by Cariello and Seoane-Sepúlveda in
\cite[Corollary 2.2]{Cariello}, the next result also includes the case where
$p\in\left(  0,1\right)  $, expanding the scope of the conclusion.

\begin{corollary}
For $p\in(0,\infty]$, the set $Z\left(  \ell_{p}\right)  $ is maximal lineable.
\end{corollary}

\begin{proof}
We simply note that the linear operator%
\begin{align*}
T\colon\ell_{\infty}  &  \rightarrow\ell_{p}\\
\left(  t_{n}\right)  _{n=2}^{\infty}  &  \mapsto%
{\textstyle\sum\limits_{n=2}^{\infty}}
t_{n}v_{n}\text{,}%
\end{align*}

defined in the proof of the previous proposition, is injective.
\end{proof}

Since the sequence $(u_n)_{n=2}^{\infty}$ of the proof of Proposition \ref{linea} is made up of elements of $c_0,$ we immediately obtain the following corollaries:
\begin{corollary}\label{ZFislineableinc0}
For $p\in(0,\infty]$, if $F$ is a closed subspace of
$\ell_{p}$ that contains $c_{0}$, then $Z\left(  F\right)  $ is $\left[
\ell_{\infty}\right]  $-lineable in $F$.
\end{corollary}
As a consequence of Corollary \ref{ZFislineableinc0} we have the following result:

\begin{corollary} \label{ZFlineableCorollary}
For $p\in(0,\infty]$, if $F$ is a closed subspace of
$\ell_{p}$ that contains $c_{0}$, then $Z\left(  F\right)  $ is lineable.
\end{corollary}
\begin{proof} Since the set $Z\left(  F\right)  $ is $\left[
\ell_{\infty}\right]  $-lineable, it is also $\left[
c_{00}\right]  $-lineable, and according to \cite{S lineability}, it is also lineable.
\end{proof}

\section{The pointwise lineability of set $Z\left(  X\right)  $ for $X=c_{0}$
or $\ell_{p}$, $p\in(0,\infty]$\label{Section3}}

In this section, we shall demonstrate, in particular, that the result from
\cite[Corollary 2.2]{Cariello} also applies within the framework of pointwise lineability (recall that a set $A$ in a vector space $X$ is
\textit{poinwise lineable} if, for every $x$ in $A$, there is an infinite dimensional
subspace $Y\subseteq X$ such that $x\in Y\subseteq A\cup\left\{  0\right\}  $. If in addition, $\dim Y=\dim X$ then $A$ is called maximal pointwise
lineable)\textbf{.}

\begin{proposition}
For $X=c_{0}$ or $\ell_{p}$ with $p\in(0,\infty]$, the set $Z\left(  X\right)
$ is maximal pointwise lineable.
\end{proposition}

\begin{proof}
Let $u_{1}=\left(  x_{n}\right)  _{n=1}^{\infty}\in X$ be a vector in
$Z\left(  X\right)  $. Consider the sequences $u_{j}=\left(  u_{n}^{\left(
j\right)  }\right)  _{n=1}^{\infty}$ for $j\geq2$, defined by $u_{n}^{\left(
j\right)  }:=j^{-n}x_{n}$ for each $n\in\mathbb{N}$ and consider the linear
operator%
\begin{align*}
T\colon\ell_{1} &  \rightarrow X\\
\left(  t_{j}\right)  _{j=1}^{\infty} &  \mapsto%
{\textstyle\sum\limits_{j=1}^{\infty}}
t_{j}u_{j}\text{.}%
\end{align*}
Due to the fact that for $(t_n)_{n=2}^{\infty} \in \ell_1,$
\begin{align*}
\left\Vert t_{j}u_{j}\right\Vert _{X}  & =\left\{
\begin{array}
[c]{ccc}%
\left(  \sum\limits_{n=1}^{\infty}\left\vert t_{j}u_{n}^{\left(  j\right)
}\right\vert ^{p}\right)  ^{1/p}\text{,} & \text{if} & p\in\left[
1,\infty\right)  \text{,}\\
\sum\limits_{n=1}^{\infty}\left\vert t_{j}u_{n}^{\left(  j\right)
}\right\vert ^{p}\text{,} & \text{if} & p\in\left(  0,1\right)  \text{,}\\
\displaystyle\sup_{n}\left\vert t_{j}u_{n}^{\left(  j\right)  }\right\vert \text{,}& \text{if} & p=\infty  \text{.}\\
\end{array}
\right.  \\
& =\left\{
\begin{array}
[c]{ccc}%
\left\vert t_{j}\right\vert \left(  \sum\limits_{n=1}^{\infty}\left\vert
u_{n}^{\left(  j\right)  }\right\vert ^{p}\right)  ^{1/p}\text{,} & \text{if}
& p\in\left[  1,\infty\right)  \text{,}\\
\left\vert t_{j}\right\vert ^{p}\sum\limits_{n=1}^{\infty}\left\vert
u_{n}^{\left(  j\right)  }\right\vert ^{p}\text{,} & \text{if} & p\in\left(
0,1\right)  \text{,}\\
\left\vert t_{j}\right\vert \displaystyle\sup_{n}\left\vert u_{n}^{\left(  j\right)
}\right\vert \text{,}& \text{if} & p=\infty  \text{.}\\
\end{array}
\right.  \\
\end{align*}
\begin{align*}
& =\left\{
\begin{array}
[c]{ccc}%
\left\vert t_{j}\right\vert \left(  \sum\limits_{n=1}^{\infty}\left\vert
j^{-n} x_n \right\vert ^{p}\right)  ^{1/p}\text{,} & \text{if}
& p\in\left[  1,\infty\right)  \text{,}\\
\left\vert t_{j}\right\vert ^{p}\sum\limits_{n=1}^{\infty}\left\vert
j^{-n} x_n \right\vert ^{p}\text{,} & \text{if} & p\in\left(
0,1\right)  \text{,}\\
\left\vert t_{j}\right\vert \displaystyle\sup_{n}\left\vert j^{-n} x_n \right\vert \text{,}& \text{if} & p=\infty  \text{.}\\
\end{array}
\right.  \\
& \leq \left\{
\begin{array}
[c]{ccc}%
\left\vert t_{j}\right\vert \left(  \sum\limits_{n=1}^{\infty}\left\vert
x_{n}\right\vert ^{p}\right)  ^{1/p}\text{,} & \text{if} & p\in\left[
1,\infty\right)  \text{,}\\
\left\vert t_{j}\right\vert ^{p}\sum\limits_{n=1}^{\infty}\left\vert
x_{n}\right\vert ^{p}\text{,} & \text{if} & p\in\left(  0,1\right)
\text{,}\\
\left\vert t_{j}\right\vert \displaystyle\sup_{n}\left\vert
x_{n}\right\vert \text{,} & \text{if} & p=\infty  \text{.}\\
\end{array}
\right.  \\
& =\left\{
\begin{array}
[c]{ccc}%
\left\vert t_{j}\right\vert \left\Vert \left( x_{n}\right)
_{n=1}^{\infty}\right\Vert _{p}\text{,} & \text{if} & p\in\left[
1,\infty\right)  \text{,}\\
\left\vert t_{j}\right\vert ^{p}\left\Vert \left(x_{n}\right)
_{n=1}^{\infty}\right\Vert _{p}\text{,} & \text{if} & p\in\left(  0,1\right)
\text{,}\\
\left\vert t_{j}\right\vert \left\Vert \left( x_{n}\right)
_{n=1}^{\infty}\right\Vert _{\infty}\text{,} & \text{if} & p=\infty  \text{.}\\
\end{array}
\right.  \text{,}%
\end{align*}
for each $j\in\mathbb{N}$, we have%
\begin{align*}
\sum\limits_{j=1}^{\infty}\left\Vert t_{j}u_{j}\right\Vert _{X}  &
\leq\left\{
\begin{array}
[c]{ccc}%
\left\Vert \left( x_{n}\right)  _{n=1}^{\infty}\right\Vert _{p}%
\sum\limits_{j=1}^{\infty}\left\vert t_{j}\right\vert \text{,} & \text{if} &
p\in\left[  1,\infty\right)  \text{,}\\
\left\Vert \left(  x_{n}\right)  _{n=1}^{\infty}\right\Vert _{p}%
\sum\limits_{j=1}^{\infty}\left\vert t_{j}\right\vert ^{p}\text{,} & \text{if}
& p\in\left(  0,1\right)  \text{,}\\
\left\Vert \left(  x_{n}\right)  _{n=1}^{\infty}\right\Vert _{\infty
}\sum\limits_{j=1}^{\infty}\left\vert t_{j}\right\vert \text{,} & \text{if} & p=\infty  \text{.}\\
\end{array}
\right.  \\
& <\infty\text{.}%
\end{align*}
Now, using the completeness of X and arguing as in Proposition \ref{linea} we conclude that $\sum\limits_{j=1}^{\infty} t_{j}u_{j}$ belongs to X, and therefore, the operator $T$ is well defined. Moreover, it is plain that $T$ is linear.
Our aim at this stage is to demonstrate that T is injective and that
\[
T\left(  \ell_{1}\right)  \subset Z\left(  X\right)  \cup\left\{
0\right\}  \text{.}%
\]
First, we aim to show that $T\left(  \ell_{1}\right)  \subset Z\left(
X\right)  \cup\left\{  0\right\}  $. For this, let%
\[
u=\sum\limits_{j=1}^{\infty}t_{j}u_{j}\in T\left(  \ell_{1}\right)
\setminus\left\{  0\right\}  \text{.}%
\]
Then, $u$ can be expressed as a sequence ot type%
\[
\left(  t_{1}x_{1}+\sum\limits_{j=2}^{\infty}t_{j}u_{1}^{\left(  j\right)
},\text{ }t_{1}x_{2}+\sum\limits_{j=2}^{\infty}t_{j}u_{2}^{\left(  j\right)
},\ldots\right)  \text{.}%
\]
Let us assume for contradiction that $u\notin Z\left(  X\right)  $. This
implies that there exists an increasing sequence of positive integers $\left(
n_{k}\right)  _{k=1}^{\infty}$ with $n_1>2$ such that%
\begin{equation}
t_{1}x_{n_{k}}+\sum\limits_{j=2}^{\infty}t_{j}u_{n_{k}}^{\left(  j\right)
}=0\text{ for each }k\in\mathbb{N}\text{.}%
\end{equation}
Since $u_{n_{k}}^{\left(  j\right)  }=j^{-n_{k}}x_{n_{k}}$, we have%
\[
\left(  t_{1}+\sum\limits_{j=2}^{\infty}t_{j}j^{-n_{k}}\right)  x_{n_{k}%
}=t_{1}x_{n_{k}}+\sum\limits_{j=2}^{\infty}t_{j}u_{n_{k}}^{\left(  j\right)
}=0\text{ for each }k\in\mathbb{N}\text{.}%
\]
Due to the presence of $u_1=\left(  x_{n}\right)  _{n=1}^{\infty}$ in $Z\left(
X\right)  $, there exists $N\in\mathbb{N}$ such that%
\[
x_{n_{k}}\neq0,\text{ for each }k\geq N\text{.}%
\]
Consequently,%
\[
t_{1}+\sum\limits_{j=2}^{\infty}t_{j}j^{-n_{k}}=0\text{ for each }k\geq
N\text{.}%
\]
Taking the limit as $k$ tends to infinity and utilizing Lemma
\ref{o lema que permite passar o limite pra dentro}, we deduce that%
\[
t_{1}=0\text{.}%
\]
Thus,%
\[
\sum\limits_{j=2}^{\infty}t_{j}j^{-n_{k}}=0\text{ for each }k\geq N\text{.}%
\]
Since%
\[
t_{2}+\sum\limits_{j=3}^{\infty}t_{j}\left(  \frac{j}{2}\right)  ^{-n_{k}%
}= 2^{n_k}\left( \sum\limits_{j=2}^{\infty}t_{j}j^{-n_{k}} \right)=0\text{ for each }k\geq N\text{,}%
\]
applying the Lemma \ref{o lema que permite passar o limite pra dentro} again we find
\[
t_{2}=0\text{.}%
\]
Proceeding recursively in this manner, we arrive at%
\[
t_{j}=0\text{ for each }j\in\mathbb{N}\text{.}%
\]
This yields a contradiction, as $u\neq0$. Therefore,%
\[
T\left(  \ell_{1}\right)  \subset Z\left(  X\right)  \cup\left\{
0\right\}  \text{.}%
\]
The arguments employed above, showing that $t_{j}=0$ for each $j\in\mathbb{N}%
$, only required the existence of an increasing sequence of positive integers
$\left(  n_{k}\right)  _{k=1}^{\infty}$ such that%
\begin{equation}
t_{1}x_{n_{k}}+\sum\limits_{j=2}^{\infty}t_{j}u_{n_{k}}^{\left(  j\right)
}=0\text{ for each }k\in\mathbb{N}\text{.}%
\end{equation}
Hence, if we assume%
\[
T\left(  \left(  t_{j}\right)  _{j=1}^{\infty}\right)  =\sum\limits_{j=1}%
^{\infty}t_{j}u_{n}=0\text{,}%
\]
then we will certainly have $\left(  t_{j}\right)  _{j=1}^{\infty}=0$, which establishes the injectivity of $T.$ Since $T\left(  1,0,0,\ldots\right)  =u_{1}$ and $\dim T(\ell_1) = \dim \ell_1 = \dim X$, the proof is complete.
\end{proof}

\begin{corollary}
For $X=c_{0}$ or $\ell_{p}$ with $p\in(0,\infty]$, the set $Z\left(  X\right)
$ is $\left(  1,\mathfrak{c}\right)  $-lineable.
\end{corollary}

\section{Topological Structure of the Complement of $Z\left(  F\right)  $ When
$F$ is a Closed Subspace of $\ell_{p}$, $p\in(0,\infty]$\label{Section4}}

In this section, we will examine results related to the notion of $(\alpha, \beta)$-spaceability in the complement of $Z(F)$, (where $F$ is a closed subspace of $\ell_p$ with $p$ in $(0, \infty]$), as well as the non-spaceability of $Z(F)$ itself. It is important to note that the non-spaceability of $Z(F)$ has already been established in \cite[Corollaries 3.4
and 4.8]{Cariello} for the case where $p \in [1, \infty]$. To achieve our objectives, we present some lemmas that will be particularly crucial for the development of Theorems \ref{TheorempabspaceabilityZX} and \ref{Theonaoabspaceability}.

\begin{lemma}
\label{Lemmadoszeros1}For $p \in (0, \infty]$, let $V$ be an infinite dimensional subspace
of $\ell_{p}$ and $E$ be a subspace of $V$ contained in $V \setminus Z\left(  V\right)
.$ If $x_{1}=\left(  x_{n}^{\left(  1\right)  }\right)
_{n=1}^{\infty}$ and $x_{2}=\left(  x_{n}^{\left(  2\right)  }\right)
_{n=1}^{\infty}$ are linearly independent vectors in $E$, then there exists an infinite set
$\mathcal{N}$ of positive integers such that $x_{n}^{\left(  1\right)  }%
=x_{n}^{\left(  2\right)  }=0$ for all $n\in\mathcal{N}$.
\end{lemma}

\begin{proof}
Let%
\[
\mathcal{N}:=\left\{  n\in\mathbb{N}:x_{n}^{\left(  k\right)  }=0\text{ for
every }k\in\left\{  1,2\right\}  \right\}  \text{.}%
\]
Fix $a\in\mathbb{K}\setminus\left\{  0\right\}  $, and let
\[
\mathcal{N}_{a}:=\left\{  n\in\mathbb{N}:ax_{n}^{\left(  1\right)  }%
=x_{n}^{\left(  2\right)  }\right\}  \text{.}%
\]
Since
\[
ax_{1}-x_{2}=\left(  ax_{n}^{\left(  1\right)  }-x_{n}^{\left(  2\right)
}\right)  _{n=1}^{\infty}\in E
\]
we can infer that $\mathcal{N}_{a}\subseteq\mathbb{N}$ is an infinite subset.
If $a$ and $b$ are two distinct nonzero scalars, and $n\in\mathcal{N}%
_{a}\cap\mathcal{N}_{b}$, then%
\[
\left(  a-b\right)  x_{n}^{\left(  1\right)  }=x_{n}^{\left(  2\right)
}-x_{n}^{\left(  2\right)  }=0\text{,}%
\]
which implies $x_{n}^{\left(  1\right)  }=x_{n}^{\left(  2\right)  }=0$,
showing that $n\in\mathcal{N}$. Thus,%
\begin{equation}
\mathcal{N}_{a}\cap\mathcal{N}_{b}\subset\mathcal{N}\text{.}
\label{aincusapdenainterNbemN0}%
\end{equation}
For the sake of contradiction, suppose $\mathcal{N}$ is finite. Since the sets
$\mathcal{N}_{a}^{\prime}:=\mathcal{N}_{a}-\mathcal{N}$ and $\mathcal{N}_{b}^{\prime}:=\mathcal{N}_{b}-\mathcal{N}$ are nonempty (remember that
$\mathcal{N}_{a}$ is infinite for each $a\in\mathbb{K}\setminus\left\{
0\right\},$) we concluded from $\left(  \text{\ref{aincusapdenainterNbemN0}}\right)
$ that $\mathcal{N}_{a}^{\prime}\cap\mathcal{N}_{b}^{\prime
}=\varnothing$. From the arbitrariness of $a$ and $b,$ then we have%
\[
\bigcup_{a\in\mathbb{K}\setminus\left\{  0\right\}  }\mathcal{N}_{a}^{\prime
}\subseteq\mathbb{N}%
\]
is uncountable, leading to a contradiction. Therefore, $\mathcal{N}$ is infinite.
\end{proof}

\begin{lemma}
\label{Lemmadoszeros2}For $p \in (0, \infty]$, let $V$ be an infinite dimensional subspace
of $\ell_{p}$, and $E$ be a subspace of $V$ contained in $V \setminus Z\left(  V\right)
.$ If $N>1$ is a positive integer, and $x_{1}=\left(
x_{n}^{\left(  1\right)  }\right)  _{n=1}^{\infty},\ldots,x_{N}=\left(
x_{n}^{\left(  N\right)  }\right)  _{n=1}^{\infty}$ are linearly independent vectors in $E$, then
there exists an increasing sequence of positive integers $\left(
n_{k}\right)  _{k=1}^{\infty}$ such that $x_{n_{k}}^{\left(  1\right)
}=\ldots=x_{n_{k}}^{\left(  N\right)  }=0$ for all $k\in\mathbb{N}$.
\end{lemma}

\begin{proof}
We proceed by induction on $N>1$. Since Lemma \ref{Lemmadoszeros1} guarantees
that the result holds for $N=2$, let us assume it is valid for $N>2$, and
prove it for $N+1$. Indeed, let%
\[
\mathcal{N}_{0}:=\left\{  n\in\mathbb{N}:x_{n}^{\left(  k\right)  }=0\text{
for each }k\in\left\{  1,\ldots,N+1\right\}  \right\}  \text{.}%
\]
Fix $a\in\mathbb{K}\setminus\left\{  0\right\}  $, and let
\[
\mathcal{N}_{a}:=\left\{  n\in\mathbb{N}:ax_{n}^{\left(  1\right)  }%
=x_{n}^{\left(  k\right)  }\text{ for each }k\in\left\{  2,\ldots,N+1\right\}
\right\}  \text{.}%
\]
Since
\[
ax_{1}-x_{k}=\left(  ax_{n}^{\left(  1\right)  }-x_{n}^{\left(  k\right)
}\right)  _{n=1}^{\infty}\in E\text{ for each }k\in\left\{  2,\ldots
,N+1\right\}  \text{,}%
\]
by the induction hypothesis, there exists an infinite set $\mathcal{N}$ such
that%
\[
ax_{n}^{\left(  1\right)  }-x_{n}^{\left(  k\right)  }=0\text{ for all }%
n\in\mathcal{N}\text{ and }k\in\left\{  2,\ldots,N+1\right\}  \text{.}%
\]
Thus, for all $n\in\mathcal{N}$, we obtain%
\[
ax_{n}^{\left(  1\right)  }=x_{n}^{\left(  2\right)  }=\ldots=x_{n}^{\left(
N+1\right)  }%
\]
showing that $\mathcal{N}_{a}$ is infinite. Taking two distinct nonzero scalars $a$ and $b$, we claim%
\[
\mathcal{N}_{a}\cap\mathcal{N}_{b}\subset\mathcal{N}_{0}\text{.}%
\]
Indeed, if $n\in\mathbb{N}_{a}\cap\mathbb{N}_{b}$, then%
\[
\left(  a-b\right)  x_{n}^{\left(  1\right)  }=x_{n}^{\left(  k\right)
}-x_{n}^{\left(  k\right)  }=0\text{,}%
\]
implying $x_{n}^{\left(  1\right)  }=x_{n}^{\left(  2\right)  }=\ldots
=x_{n}^{\left(  N+1\right)  }=0$, showing that $n\in\mathcal{N}_{0}$.
Similarly to what we did in the proof of Lemma \ref{Lemmadoszeros1}, we
conclude that $\mathcal{N}_{0}$ is infinite, which verifies the case $N+1$.
Therefore, the result is complete.
\end{proof}

We will now examine the structure of the complement of $Z\left(  F\right)  $,
when $F$ is closed subespace of $\ell_p,$ $p \in (0, \infty],$ with a focus on the notions of $\left(  \alpha
,\beta\right)  $-spaceability. To this end, let us begin with the following concept as
introduced in \cite{Drewnowski}.

\begin{definition}
A sequence $\left(  u_{n}\right)  _{n=1}^{\infty}$ of elements in a
topological vector space $X$ is said to be topologically linearly independent
if, for any sequence $\left(  c_{n}\right)  _{n=1}^{\infty}\in\mathbb{K}%
^{\mathbb{N}}$ with $\sum_{n=1}^{\infty}c_{n}u_{n}=0$, it follows that
$\left(  c_{n}\right)  _{n=1}^{\infty}=0$.
\end{definition}

Based on this definition, if $\mathcal{S}\neq\left\{  0\right\}  $ is a
subspace of $\mathbb{K}^{\mathbb{N}}$ then we will say that a sequence
$\left(  u_{n}\right)  _{n=1}^{\infty}$ of elements of a topological vector
space $X$ is $\mathcal{S}$-topologically linearly independent in\textbf{ }$X$
(or $\mathcal{S}$-independent) if for each sequence $\left(  c_{n}\right)
_{n=1}^{\infty}\in\mathcal{S}$ with $\sum_{n=1}^{\infty}c_{n}u_{n}=0$, we have
$\left(  c_{n}\right)  _{n=1}^{\infty}=0$.

\bigskip

Within this perspective, as shown by Drewnowski $\left(  \text{see
\cite{Drewnowski}}\right)  $ the following result establishes a crucial link
between linear $\ell_{\infty}$-independence and linear independence, which
will be instrumental in proving the subsequent theorems.

\bigskip

\begin{proposition}
\label{Proposition sequence l-independent} Assume that $\left(  u_{n}\right)
_{n=1}^{\infty}$ is a linearly independent sequence in a Hausdorff topological
vector space $X$. Then there exists a sequence $(\lambda_{n})_{n=1}^{\infty}$ of positive real numbers such that $\left(
\lambda_{n}u_{n}\right)  _{n=1}^{\infty}$ is $\ell_{\infty}$-independent.
\end{proposition}

To facilitate writing, from now on, for each positive integer $r$, we will denote
$$
W_r = \{ (x_n)_{n=1}^{\infty} \in \mathbb{K}^{\mathbb{N}} \; ; \; x_1 = \cdots = x_r = 0 \}.
$$

The next lemma will be of utmost importance as it will provide us with the step-by-step procedure for proceeding in the subsequent results, concerning both positive and negative outcomes.
\begin{lemma}
\label{Lema das m primeiras} For $p \in (0, \infty]$, if $F$ is a closed infinite dimensional subspace of $\mathbb{\ell}_{p}$, then there exists a sequence of linearly independent vectors $\left( y_{k} \right)_{k=1}^{\infty}$ in $F\setminus Z(F)$ which also makes $F\setminus Z(F)$ a lineable set.
\end{lemma}

\begin{proof}
Let $m_{1}$ be a positive integer with $m_{1}>1$ and consider linearly
independent vectors $(x_{n}^{\left(  1\right)  })_{n=1}^{\infty},\ldots
,(x_{n}^{\left(  m_{1}\right)  })_{n=1}^{\infty}$ in $F$.  Since the $m_{1}$ vectors
$(x_{n}^{\left(  1\right)  })_{n=1}^{m_{1}-1},\ldots,(x_{n}^{\left(
m_{1}\right)  })_{n=1}^{m_{1}-1}$  are linearly dependent
in $\mathbb{K}^{m_{1}-1}$, there exists a non-trivial linear combination of them resulting in the null vector. Without
loss of generality, let us suppose that there exist scalars $a_{2}%
,\ldots,a_{m_{1}}$ such that%
\[
\left(  x_{n}^{\left(  1\right)  }\right)  _{n=1}^{m_{1}-1}+\sum
\limits_{i=2}^{m_{1}}a_{i}\left(  x_{n}^{\left(  i\right)  }\right)
_{n=1}^{m_{1}-1}=0\text{.}%
\]
Thus, we obtain that the vector $v_{1}:=(x_{n}^{\left(
1\right)  })_{n=1}^{\infty}+\sum_{i=2}^{m_{1}}a_{i}(x_{n}^{\left(  i\right)
})_{n=1}^{\infty}\in W_{m_1-1} \cap \left[F\setminus\left\{  0\right\} \right].$  Let
\[
r_1=\min\ \{ n \geq 0: x_{m_{1}+n}^{\left(  1\right)  }+\sum_{i=2}^{m_{1}}a_{i}x_{m_{1}+n%
}^{\left(  i\right)}\neq0 \} \text{.}%
\]
\newline Similarly, consider $z_{1}:=\left(  z_{n}^{\left(  1\right)
}\right)  _{n=1}^{\infty} \in W_{m_2} \cap \left[F\setminus\left\{  0\right\} \right],$ where 
$m_{2}:=m_{1}+r_{1}+1$ and let $r_{2}$ be the smallest
positive integer $n$ such that $z_{m_{2}+n}^{\left(  1\right)  }\neq0$.
Define $\varepsilon_{1}=\frac{1}{2}$ and let $u_{1}=\left(  u_{n}^{\left(  1\right)
}\right)  _{n=1}^{\infty}$ be a vector belonging to $W_{m_{2}+r_{2}-1} \cap \left[F\setminus\left\{  0\right\}\right].$ Define $m_{3}:=m_{2}+r_{2}+1$ and let $\lambda_{1}$ be a
non-zero scalar such that%
\[
|\lambda_{1}|>\max\left\{  \frac{ \left\vert u_{m_{2}+r_{2}}^{\left(  1\right)  } \right\vert}%
{\left\vert z_{m_{2}+r_{2}}^{\left(  1\right)  } \right\vert},\frac{\left\vert v_{m_{2}+r_{2}%
}^{\left(  1\right)  }\right\vert +\varepsilon_{1} \left\vert u_{m_{2}+r_{2} }^{\left(
1\right)  }\right\vert}{\varepsilon_{1} \left\vert z_{m_{2}+r_{2}}^{\left(  1\right)  } \right\vert }\right\}
\text{.}%
\]
By considering the vector $v_{2}:=\lambda_{1}z_{1}-u_{1}$, we have:

\begin{itemize}
\item $v_{m_{2}+r_{2}}^{\left(  2\right)  }=\lambda_{1}z_{m_{2}+r_{2}%
}^{\left(  1\right)  }-u_{m_{2}+r_{2}}^{\left(  1\right)  }\neq0,$ in fact,
$\lambda_{1}z_{m_{2}+r_{2}}^{\left(  1\right)  }-u_{m_{2}+r_{2}}^{\left(
1\right)  }=0$ if and only if $\lambda_{1}=\frac{u_{m_{2}+r_{2}}^{\left(
1\right)  }}{z_{m_{2}+r_{2}}^{\left(  1\right)  }},$

\item $v_{n}^{\left(  2\right)  }=0$ whenever $n\in\left\{  1,\ldots
,m_{2}-1\right\}  $,

\item $\left\vert v_{m_{2}+r_{2}}^{\left(  1\right)  }\right\vert
\leq\varepsilon_{1}\left\vert v_{m_{2}+r_{2}}^{\left(  2\right)  }\right\vert
$.
\end{itemize}
Consider now $z_{2}:=\left(  z_{n}^{\left(  2\right)
}\right)  _{n=1}^{\infty} \in W_{m_3} \cap \left[F\setminus\left\{  0\right\} \right],$ where
$m_{3}=m_2+r_2+1$ and let $r_{3}$ be the smallest
positive integer $n$ such that $z_{m_{3}+n}^{\left(  2\right)  }\neq0$.
Define $\varepsilon_{2}=\frac{1}{2^2}$ and let $u_{2}=\left(  u_{n}^{\left(  2\right)
}\right)  _{n=1}^{\infty}$ be a vector in $ W_{m_3+r_3-1} \cap \left[F\setminus\left\{  0\right\} \right].$ Let $\lambda_{2}$ be a
non-zero scalar such that%
\[
|\lambda_{2}|>\max\left\{  \frac{ \left\vert u_{m_{3}+r_{3}}^{\left(  2\right)  }\right\vert}  %
{ \left\vert z_{m_{3}+r_{3}}^{\left(  2\right)  }\right\vert},\frac{\left\vert v_{m_{3}+r_{3}%
}^{\left(  2\right)  }\right\vert + \left\vert v_{m_{3}+r_{3}%
}^{\left(  1\right)  }\right\vert+\varepsilon_{2} \left\vert u_{m_{3}+r_{3}}^{\left(
2\right)  }\right\vert }{\varepsilon_{2} \left\vert z_{m_{3}+r_{3}}^{\left(  2\right)  } \right\vert }\right\}
\text{.}%
\]
By considering the vector $v_{3}:=\lambda_{2}z_{2}-u_{2}$, we have:

\begin{itemize}
\item $v_{m_{3}+r_{3}}^{\left(  3\right)  }=\lambda_{2}z_{m_{3}+r_{3}%
}^{\left(  2\right)  }-u_{m_{3}+r_{3}}^{\left(  2\right)  }\neq0,$ in fact,
$\lambda_{2}z_{m_{3}+r_{3}}^{\left(  2\right)  }-u_{m_{3}+r_{3}}^{\left(
2\right)  }=0$ if and only if $\lambda_{2}=\frac{u_{m_{3}+r_{3}}^{\left(
2\right)  }}{z_{m_{3}+r_{3}}^{\left(  2\right)  }},$

\item $v_{n}^{\left(  3\right)  }=0$ whenever $n\in\left\{  1,\ldots
,m_{3}-1\right\}  $,

\item $\left\vert v_{m_{3}+r_{3}}^{\left(  1\right)  }\right\vert +\left\vert
v_{m_{3}+r_{3}}^{\left(  2\right)  }\right\vert \leq\varepsilon_{2}\left\vert
v_{m_{3}+r_{3}}^{\left(  3\right)  }\right\vert $.
\end{itemize}
This construction allows us to obtain, inductively, an increasing sequence of
positive integers $\left(  m_{k}\right)_{k=1}^{\infty}$ and a sequence of vectors $v_{k}:=\left(
v_{n}^{\left(  k\right)  }\right)  _{n=1}^{\infty}$ in $F$ satisfying:

\begin{itemize}
\item $v_{m_{k+1}+r_{k+1}}^{\left(  k+1\right)  }\neq0,$

\item $v_{n}^{\left(  k+1\right)  }=0$ whenever $n\in\left\{  1,\ldots
,m_{k+1}-1\right\},$

\item $\sum_{i=1}^{k}\left\vert v_{m_{k+1}+r_{k+1}}^{\left(  i\right)
}\right\vert \leq\varepsilon_{k}\left\vert v_{m_{k+1}+r_{k+1}}^{\left(
k+1\right)  }\right\vert, $
\end{itemize}

where $\varepsilon_k = \frac{1}{2^k},$ for each $k\in \mathbb{N}.$ Without loss of generality, let us assume that $\vert \vert v_k \vert\vert_p=1,$ for each $k \in \mathbb{N}.$ In the next step, fix $k$ in $\mathbb{N}$, and define the vectors
\[
x_{0,k} = v_k \, \, \, \text{and} \, \, \, x_{1,k}=v_{k}-\left(  \frac{v_{m_{k+1}+r_{k+1}}^{\left(  k\right)  }%
}{v_{m_{k+1}+r_{k+1}}^{\left(  k+1\right)  }}\right)  v_{k+1}\text{.}%
\]
Denote $x_{1,k}$ by
$\left(x_{1,n}^{ (k) } \right)_{n=1}^{\infty}.$ Note that, for each $n \in \mathbb{N},$
\[
\left\vert x_{1,n}^{(k)}\right\vert \leq\left\vert v_{n}^{(k)}\right\vert +\left\vert
\frac{v_{m_{k+1}+r_{k+1}}^{\left(  k\right)  }}{v_{m_{k+1}+r_{k+1}}^{\left(
k+1\right)  }}\right\vert \left\vert v_{n}^{(k+1)}\right\vert \leq \left\vert
v_{n}^{(k)}\right\vert +\left\vert v_{n}^{(k+1)}\right\vert \text{.}%
\]

The vector $x_{1,k}\in F$  also satisfies:

\begin{itemize}
\item $x_{1,m_{k}+r_{k}}^{\left(  k\right)  }=v_{m_{k}+r_{k}}^{\left(
k\right)  }$

\item $x_{1,n}^{\left(  k\right)  }=0$ whenever $n\in\left\{  1,\ldots
,m_{k}-1\right\}  $ and $x_{1,m_{k+1}+r_{k+1}}^{\left(  k\right)  }=0$.

\item $\left\Vert x_{1,k}-x_{0,k}\right\Vert _{p}=\left\{
\begin{array}
[c]{ccc}%
\text{ }\left\vert \frac{v_{m_{k+1}+r_{k+1}}^{\left(  k\right)  }}%
{v_{m_{k+1}+r_{k+1}}^{\left(  k+1\right)  }}\right\vert ^{p}\text{,} &
\text{if} & p\in\left(  0,1\right)  \text{,}\\
&  & \\
\left\vert \frac{v_{m_{k+1}+r_{k+1}}^{\left(  k\right)  }}{v_{m_{k+1}+r_{k+1}%
}^{\left(  k+1\right)  }}\right\vert \text{,} & \text{if} & p\in\left[
1,\infty\right]  \text{.}%
\end{array}
\right.  \leq\left\{
\begin{array}
[c]{ccc}%
\left(  \varepsilon_{k}\right)  ^{p}\text{,} & \text{if} & p\in\left(
0,1\right)  \text{,}\\
\varepsilon_{k}\text{,} & \text{if} & p\in\left[  1,\infty\right]  \text{.}%
\end{array}
\right.  $.
\end{itemize}

Thus, by considering the vector%
\[
x_{2,k}:=x_{1,k}-\left(  \frac{x_{1,m_{k+2}+r_{k+2}}^{\left(  k\right)  }%
}{v_{m_{k+2}+r_{k+2}}^{\left(  k+2\right)  }}\right)  v_{k+2}%
\]
and denoting $x_{2,k}$ by $\left(  x_{2,n}^{\left(  k\right)  }\right)
_{n=1}^{\infty}$, we have:

\begin{itemize}
\item $x_{2,m_{k}+r_{k}}^{\left(  k\right)  }=x_{1,k}=v_{m_{k}+r_{k}}^{\left(
k\right)  }$,

\item $x_{2,n}^{\left(  k\right)  }=0$ whenever $n\in\left\{  1,\ldots
,m_{k}-1\right\}  $ and $x_{2,m_{k+1}+r_{k+1}}^{\left(  k\right)
}=x_{2,m_{k+2}+r_{k+2}}^{\left(  k\right)  }=0$.
\end{itemize}
Due to the fact that%
\[
\left\vert x_{1,m_{k+2}+r_{k+2}}^{\left(  k\right)  }\right\vert
\leq\left\vert v_{m_{k+2}+r_{k+2}}^{\left(  k\right)  }\right\vert +\left\vert
v_{m_{k+2}+r_{k+2}}^{\left(  k+1\right)  }\right\vert \leq\varepsilon
_{k+1}\left\vert v_{m_{k+2}+r_{k+2}}^{\left(  k+2\right)  }\right\vert \text{,}%
\]
we also have

\begin{itemize}
\item $\left\Vert x_{2,k}-x_{1,k}\right\Vert _{p}=\left\{
\begin{array}
[c]{ccc}%
\text{ }\left\vert \frac{x_{1,m_{k+2}+r_{k+2}}^{\left(  k\right)  }%
}{v_{m_{k+2}+r_{k+2}}^{\left(  k+2\right)  }}\right\vert ^{p}\text{,} &
\text{if} & p\in\left(  0,1\right)  \text{,}\\
&  & \\
\left\vert \frac{x_{1,m_{k+2}+r_{k+2}}^{\left(  k\right)  }}{v_{m_{k+2}%
+r_{k+2}}^{\left(  k+2\right)  }}\right\vert \text{,} & \text{if} &
p\in\left[  1,\infty\right]  \text{.}%
\end{array}
\right.  \leq\left\{
\begin{array}
[c]{ccc}%
\left(  \varepsilon_{k+1}\right)  ^{p}\text{,} & \text{if} & p\in\left(
0,1\right)  \text{,}\\
\varepsilon_{k+1}\text{,} & \text{if} & p\in\left[  1,\infty\right]  \text{.}%
\end{array}
\right.  $.
\end{itemize}

Proceeding in this manner, we obtain a sequence $\left(  x_{j,k}\right)  _{j=1}^{\infty}$ in $F$ satisfying properties as those above. As for $m>r$ we have
\begin{align}
\left\Vert x_{m,k}-x_{r,k}\right\Vert _{p}  &  \leq\left\Vert x_{m,k}%
-x_{m-1,k}\right\Vert _{p}+\left\Vert x_{m-1,k}-x_{m-2,k}\right\Vert
_{p}+\cdots+\left\Vert x_{r+1,k}-x_{r,k}\right\Vert _{p}%
\label{CauchysequenceinZF}\\
&  \leq\varepsilon_{k+m-1}+\varepsilon_{k+m-2}+\cdots+\varepsilon_{k+r}%
\leq\varepsilon_{k+r}\nonumber
\end{align}

we can infer that
\[
\lim_{j\rightarrow\infty}x_{j,k}=y_{k}\in F\text{.}%
\]
By construction the $k$-th vector $y_{k}$ $($denoted by $\left(  y_{n,k}\right)
_{n=1}^{\infty})$ satisfies:
\[
y_{m_{k}+r_{k},k}=\lim_{j\rightarrow\infty}x_{j,m_{k}+r_{k}}^{\left(
k\right)  }=v_{m_{k}+r_{k}}^{\left(  k\right)  }\neq0.
\]
Also, as $x_{j,n}^{(k)}=0$ for all $n\in\{1,\ldots,m_{k}-1=m_{k-1}%
+r_{k-1}\}\cup\{m_{k+1}+r_{k+1},\ldots,m_{k+j}+r_{k+j}\}$, we have
\begin{equation}
y_{m_{i}+r_{i},k}=\lim_{j\rightarrow\infty}x_{j,m_{i}+r_{i}}^{\left(
k\right)  }=0\,\,\text{whenever}\,\,i\neq
k.\label{QaundosaodiferentesdazeronoLema}%
\end{equation}
It is plain that the sequence $\left(  y_{k}\right)  _{k=1}^{\infty}$ is linearly independent. If we assume, without loss of generality, that $\left\Vert y_{k}\right\Vert
_{p}=1$ for all $k\in\mathbb{N}$ and that these vectors form an $\ell_{\infty
}$-independent sequence $\left(  \text{see Proposition
\ref{Proposition sequence l-independent}}\right)  $, we have that the series $%
{\textstyle\sum\limits_{k=1}^{\infty}}
\left\Vert t_{k}2^{-k}y_{2k}\right\Vert _{p}$ converges whenever $\left(
t_{k}\right)  _{k=1}^{\infty}\in\ell_{\infty}$ and that $%
{\textstyle\sum\limits_{k=1}^{\infty}}
t_{k}2^{-k}y_{2k}=0\Rightarrow t_{k}=0$ for each $k\in\mathbb{N}$. This shows
that the linear operator%
\begin{align*}
T\colon\ell_{\infty} &  \rightarrow F\\
\left(  t_{k}\right)  _{k=1}^{\infty} &  \mapsto%
{\textstyle\sum\limits_{k=1}^{\infty}}
t_{k}2^{-k}y_{2k}%
\end{align*}
is well defined and injective. It remains to show now that $T\left(
\ell_{\infty}\right)  \subset F\setminus Z\left(  F\right)  $. Indeed, due to
the fact that
\[%
{\textstyle\sum\limits_{k=1}^{\infty}}
t_{k}2^{-k}y_{2k}=\left(
{\textstyle\sum\limits_{k=1}^{\infty}}
t_{k}2^{-k}y_{1,2k},%
{\textstyle\sum\limits_{k=1}^{\infty}}
t_{k}2^{-k}y_{2,2k},%
{\textstyle\sum\limits_{k=1}^{\infty}}
t_{k}2^{-k}y_{3,2k},\ldots\right)  \text{,}%
\]
we can invoke $\left(  \text{\ref{QaundosaodiferentesdazeronoLema}}\right)  $
to obtain
\[
y_{m_{i}+r_{i},2k}=\lim_{j\rightarrow\infty}x_{j,m_{i}+r_{i}}^{\left(
2k\right)  }=0\,\text{whenever }k\in\mathbb{N}\text{ and }i\text{ is an odd
positive integer.}%
\]
Thus,%
\[%
{\textstyle\sum\limits_{k=1}^{\infty}}
t_{k}2^{-k}y_{m_{i}+r_{i},2k}=0\text{ for each odd positive integer }i\text{,}%
\]
and indeed each expression $%
{\textstyle\sum\limits_{k=1}^{\infty}}
t_{k}2^{-k}y_{2k}$ in the context of the operator $T$ will be a sequence with
infinitely many zeros. This shows that $T\left(  \ell_{\infty}\right)  \subset
F\setminus Z\left(  F\right)  $ as we wanted. Therefore, the proof is complete.
\end{proof}

The above result provide the information that it is always possible to
obtain a vector in any infinite dimensional subspace of $\mathbb{K}^{\mathbb{N}}$ whose first $m$ coordinates are zero. Some ideas for the proof
of the same are due to the authors in \cite{Cariello}.

\bigskip

With this foundational understanding, we are now prepared to establish the main results this section.

\begin{theorem}\label{TheorempabspaceabilityZX}
For $p \in (0, \infty]$, if $F$ is an infinite dimensional closed subspace of $\ell_{p}$,
then the set $F\setminus Z\left(  F\right)  $ is $\left(  \alpha
,\mathfrak{c}\right)  $-spaceable for each $\alpha<\aleph_{0}$.
\end{theorem}

\begin{proof}
Fix $\alpha<\aleph_{0}$ and let $E$ be a subspace of $F$ such that $E\subseteq
F\setminus Z(F)$ and $\dim(E)=\alpha$. The existence of such a subspace $E$ is
guaranteed by the Lemma \ref{Lema das m primeiras}. Let $\{w_{k}%
\}_{k=1}^{\alpha}$ be a basis for $E$ and denote $w_{k}$ by $(w_{n}%
^{(k)})_{n=1}^{\infty}$. We will start this proof by choosing positive
integers $s_{1}<s_{2}<\ldots<s_{\alpha}$ and a basis $\{w_{1},\ldots
,w_{\alpha}\}$ for $E$ such that for each $k\in\{1,\ldots,\alpha\}$,

\begin{itemize}
\item $w_{s_{k}}^{(k)}\neq0$,

\item $w_{s}^{(k)}=0$ whenever $s<s_{k}$.
\end{itemize}

We proceed as follows:

1. Define:
\[
s_{1} = \min\left\{  \min\left\{  n \in\mathbb{N} : w_{n}^{(1)} \neq0
\right\}  , \min\left\{  n \in\mathbb{N} : w_{n}^{(2)} \neq0 \right\}  ,
\ldots, \min\left\{  n \in\mathbb{N} : w_{n}^{(\alpha)} \neq0 \right\}
\right\}  .
\]

Without loss of generality, assume that $s_{1}=\min\left\{  n\in
\mathbb{N}:w_{n}^{(1)}\neq0\right\}  $.

2. For $k\in\{2,\ldots,\alpha\}$, consider the vector:
\[
x_{k}=%
\begin{cases}
w_{k}, & \text{if }w_{s_{1}}^{(k)}=0,\\
w_{k}-\frac{w_{s_{1}}^{(k)}}{w_{s_{1}}^{(1)}}w_{1}, & \text{otherwise}.
\end{cases}
\]

Also, define%
\[
x_{1}=w_{1}\text{.}%
\]
The minimality of $s_{1}$ ensures that $w_{s}^{(k)}=0$ whenever $s<s_{1}$.
Moreover, since for $k\in\{2,\ldots,\alpha\}$, we have $x_{s_{1}}^{(k)}=0$
(denote $x_{k}$ by $(x_{n}^{(k)})_{n=1}^{\infty}$), we can infer that
$x_{2},\ldots,x_{k}\in W_{s_{1}}$. This construction also allows us to infer
that the set $\{x_{k}\}_{k=1}^{\alpha}$ forms a basis for $E$.

3. Proceeding, define:
\[
s_{2}=\min\left\{  \min\left\{  n\in\mathbb{N}:x_{n}^{(2)}\neq0\right\}
,\ldots,\min\left\{  n\in\mathbb{N}:x_{n}^{(\alpha)}\neq0\right\}  \right\}
\text{.}%
\]

Without loss of generality, assume that $s_{2}=\min\left\{  n\in
\mathbb{N}:x_{n}^{(2)}\neq0\right\}  $.

4. For $k\in\{3,\ldots,\alpha\}$, consider the vector:
\[
z_{k}=%
\begin{cases}
x_{k}, & \text{if }x_{s_{2}}^{(k)}=0,\\
x_{k}-\frac{x_{s_{2}}^{(k)}}{x_{s_{2}}^{(2)}}x_{2}, & \text{otherwise}.
\end{cases}
\]
Also, define%
\[
z_{1}=x_{1}\text{ \ and \ }z_{2}=x_{2}\text{.}%
\]

The minimality of $s_{2}$ ensures that $x_{s}^{(k)}=0$ whenever $s<s_{2}$.
Moreover, since for $k\in\{3,\ldots,\alpha\}$, we have $z_{s_{2}}^{(k)}=0$
(denote $z_{k}$ by $(z_{n}^{(k)})_{n=1}^{\infty}$), we can infer that
$z_{3},\ldots,z_{k}\in W_{s_{2}}$. Furthermore, $z_{2}=(z_{n}^{(2)}%
)_{n=1}^{\infty}\in W_{s_{2}-1}\setminus W_{s_{2}}$, since $x_{s_{2}}%
^{(2)}\neq0$ and $z_{1}=(z_{n}^{(1)})_{n=1}^{\infty}\in W_{s_{2}-1}$. This
construction also allows us to infer that the set $\{z_{k}\}_{k=1}^{\alpha}$
forms a basis for $E$.

This process allows us to obtain positive integers $s_{1}<s_{2}<\ldots
<s_{\alpha}$ and consider a basis $\{w_{1},\ldots,w_{\alpha}\}$ in $E$ such
that for each $k\in\{1,\ldots,\alpha\}$, $w_{s_{k}}^{(k)}\neq0$ and
$w_{s}^{(k)}=0$ whenever $s<s_{k}$, desired.

Now, using the hypothesis that $E\subseteq F\setminus Z(F)$, we can invoke
Lemma \ref{Lemmadoszeros2}, and obtain an increasing sequence of positive
integers $(n_{i})_{r=1}^{\infty}$ such that%
\[
w_{n_{i}}^{(k)}=0\text{ forall }i\in\mathbb{N}\text{ and each }k\in
\{1,\ldots,\alpha\}\text{.}%
\]

To construct a subspace of dimension $\mathfrak{c}$ in $F\setminus Z(F)$
containing $E$, proceed as follows:

1. Take $m_{1}=s_{\alpha}+1$ and fix $v_{1}=(v_{n}^{(1)})_{n=1}^{\infty}$ in
$W_{m_{1}}\cap\lbrack F\setminus\{0\}]$.

2. Consider tamb\'{e}m $r_{1}=\min\left\{  n\in\mathbb{N}:v_{m_{1}+n}%
^{(1)}\neq0\right\}  $.

3. Fix a positive integer $m_{2}$ such that $m_{2}>\max\{s_{\alpha+1}%
,m_{1}+r_{1}\}$ and fix $z_{1}=(z_{n}^{(1)})_{n=1}^{\infty}$ in $W_{m_{2}}%
\cap\lbrack F\setminus\{0\}]$.

4. Then, take $r_{2}=\min\left\{  n\in\mathbb{N}:z_{m_{2}+n}^{(1)}%
\neq0\right\}  $, define $\varepsilon_{1}=1/2$ and consider $u_{1}%
=(u_{n}^{(1)})_{n=1}^{\infty}$ in $W_{m_{2}+r_{2}-1}\cap\lbrack F\setminus
\{0\}]$.

5. Furthermore, choose $\lambda_{1}$ in $\mathbb{K}\setminus\{0\}$ such that
\[
|\lambda_{1}|>\max\left\{  \frac{|u_{m_{2}+r_{2}}^{(1)}|}{|z_{m_{2}+r_{2}%
}^{(1)}|},\frac{|v_{m_{2}+r_{2}}^{(1)}|+\varepsilon_{1}|u_{m_{2}+r_{2}}%
^{(1)}|}{\varepsilon_{1}|z_{m_{2}+r_{2}}^{(1)}|}\right\}  .
\]

6. Now, define the vector $v_{2}:=\lambda_{1}z_{1}-u_{1}$ (denote $v_{2}$ by
$(v_{n}^{(2)})_{n=1}^{\infty}$).\newline

The above construction ensures that $v_{m_{2} + r_{2}}^{(2)} = \lambda_{1}
z_{m_{2} + r_{2}}^{(1)} - u_{m_{2} + r_{2}}^{(1)} \neq0 $ and $v_{n}^{(2)} = 0
$ whenever $n \in\{1, \ldots, n_{1}, \ldots, n_{\alpha+1}, \ldots, m_{2} - 1\}
$ and that $|v_{m_{2} + r_{2}}^{(1)}| \leq\varepsilon_{1} |v_{m_{2} + r_{2}%
}^{(2)}| $.

Proceed similarly as in the proof of the Lemma \ref{Lema das m primeiras},
(choosing $m_{i+1}>\max\{n_{i},m_{i}+r_{i}\}$ for each $i>1$), we also obtain
a sequence $y_{1}=(y_{n}^{(1)})_{n=1}^{\infty},y_{2}=(y_{n}^{(2)}%
)_{n=1}^{\infty},\ldots$ in $F\setminus Z(F)$ such that for each
$k\in\mathbb{N}$:

\begin{itemize}
\item $y_{m_{k}+r_{k}}^{\left(  k\right)  }\neq0$,

\item $y_{m_{i}+r_{i}}^{\left(  k\right)  }=0$ for each $i\neq k$,

\item $y_{n_{i}}^{\left(  k\right)  }=0$ for each $i\in\mathbb{N}$.
\end{itemize}

Defining $t_{1}=\ldots=t_{\alpha}=0$ and also $t_{i}=r_{i},s_{i}=m_{i},$ and
$w_{i}=y_{i}$, for $i>\alpha$ (to unify the notation with the basis $\left\{
w_{1},\ldots,w_{\alpha}\right\}  $ in $E$), we can conclude that for each
$k\in\mathbb{N}$:

\begin{itemize}
\item $w_{s_{k}+t_{k}}^{\left(  k\right)  }\neq0$,

\item $w_{s_{i}+t_{i}}^{\left(  k\right)  }=0$ whenever $k>i$,

\item $w_{n_{i}}^{\left(  k\right)  }=0$ for each $i\in\mathbb{N}$.
\end{itemize}

This shows that the vectors $(w_{k})_{k=1}^{\infty}$, besides belonging to
$F\setminus Z(F)$, are also linearly independent. Assume, without loss of
generality, that $\left\Vert w_{k}\right\Vert _{p}=1$ for each $k\in
\mathbb{N}$, and that $(w_{k})_{k=1}^{\infty}$ form an $\ell_{\infty}%
$-independent sequence $\left(  \text{see Proposition
\ref{Proposition sequence l-independent}}\right)  $. Now, consider the linear
operator%
\begin{align*}
T\colon\ell_{\infty} &  \rightarrow F\\
\left(  t_{k}\right)  _{k=1}^{\infty} &  \mapsto%
{\textstyle\sum\limits_{k=1}^{\alpha}}
t_{k}2^{-k}w_{k}\text{.}%
\end{align*}
Since $%
{\textstyle\sum\limits_{k=\alpha+1}^{\infty}}
\left\Vert t_{k}2^{-k}w_{k}\right\Vert _{p}<\infty$, the operator $T$ is well
defined. Moreover, due to the fact that $(w_{k})_{k=1}^{\infty}$ is
$\ell_{\infty}$-indepedent, $T$ also is injective. This shows that $\dim
T\left(  \ell_{\infty}\right)  =\dim\ell_{\infty}=\mathfrak{c}$. \\

We claim that the closure of $T\left(  \ell_{\infty}\right)  $ in $F$ $\left(
\text{denoted by }\overline{T\left(  \ell_{\infty}\right)  }\right)$ is contained in $F \setminus Z\left(  F\right)  $. For this, let $u\in\overline{T\left(
\ell_{\infty}\right)  }\setminus\left\{  0\right\}  $. Then there exist
sequences $\left(  t_{k}^{\left(  j\right)  }\right)  _{k=1}^{\infty}$ in
$\ell_{\infty}$, $j\in\mathbb{N}$, such that%
\[
u=\lim_{j\rightarrow\infty}T\left(  \left(  t_{k}^{\left(  j\right)  }\right)
_{k=1}^{\infty}\right)  \text{ in }F\text{.}%
\]
That is,
\[
u=\lim_{j\rightarrow\infty}T\left(  \left(  t_{k}^{\left(  j\right)  }\right)
_{k=1}^{\infty}\right)  =\lim_{j\rightarrow\infty}\left(
{\textstyle\sum\limits_{k=1}^{\infty}}
t_{k}^{\left(  j\right)  }2^{-k}w_{k} \right)  \text{.}%
\]
Due to the fact that $\sum_{k=1}^{\infty}t_{k}^{\left(
j\right)  }2^{-k}w_{n_{i}}^{\left(  k\right)  }=0$ for each
$i,j\in\mathbb{N}$, we get%
\[
\lim_{j\rightarrow\infty}\left(
{\textstyle\sum\limits_{k=1}^{\infty}}
t_{k}^{\left(  j\right)  }2^{-k}w_{n_{i}}^{\left(  k\right)  }\right)
=0\text{ for each }i\in\mathbb{N}\text{,}%
\]
which shows that $u$ belongs to $F \setminus Z(F).$ Therefore, since $E \subset\overline{T\left(  \ell_{1}\right)  } \subset F \setminus Z(F),$ and \( \overline{T(\ell_1)} \) is a subspace of dimension $\mathfrak{c}$ in $ F $ $(\text{due to the fact that } T \text{ is injective}),$ the result is done.
\end{proof}
Since $F \setminus Z(F)$ is closed by scalar multiplication, we have the following corollary:
\begin{corollary}
\label{CorollaryinZFpointwise}For $p \in (0, \infty]$, if $F$ is an infinite dimensional
closed subspace of $\ell_{p}$, then the set $F\setminus Z\left(  F\right)  $
is $\alpha$-pointwise lineable for each $\alpha\geq\aleph_{0}$.
\end{corollary}
\section{Negative results}\label{Section5}

	In this section, we present two fundamental results that we believe will significantly enhance the understanding of the structure of \( Z(F) \) and its complement, in the case where \( F \) is a closed subspace of \( \ell_p \) with \( p \in (0, \infty] \). Based on Lemma \ref{Lema das m primeiras} developed in the previous section, we explore the issue of non spaceability.
	
	In \cite{Cariello}, Cariello and Seoane-Sepúlveda proved that if \( F \) is an infinite dimensional subspace of \( \ell_p \) with \( p \in (0, \infty] \), then \( Z(F) \) is not spaceable. In Theorem \ref{nonspaceabilityZF01}, based on Lemma \ref{Lema das m primeiras}, we provide another proof for this fact and, at the same time, show that the result also holds for \( p \in (0, 1) \). This result is significant as it extends the understanding of the structure of \( Z(F) \) to values of \( p \) not previously addressed in the literature.
	
	Furthermore, we explore the issue of \((\alpha, \beta)\)-non-spaceability. We demonstrate that, for \( p \in (0,\infty] \) and any \( \alpha \geq \aleph_0 \), if \( F \) is an infinite dimensional closed subspace of \( \ell_p \) that contains \( c_0 \), then the set \( F \setminus Z(F) \) is not \((\alpha, \mathfrak{c})\)-spaceable. This theorem not only generalizes but also complements previous results, providing a more comprehensive and robust understanding of the behavior of sequence spaces.
	
	Below, we present the detailed proofs of the mentioned results.

\begin{theorem}
\label{nonspaceabilityZF01} For $p\in\left(  0, \infty\right]  $, if $F$ is an
infinite dimensional closed subspace of $\ell_{p}$, then the set $Z\left(
F\right)  $ not is spaceable.
\end{theorem}

\begin{proof}
Assume, for the sake of contradiction, that \( Z(F) \) is spaceable. Let \( F_1 \) be a closed infinite-dimensional subspace of \( F \) in \( Z(F) \cup \{0\} \). By Lemma \ref{Lema das m primeiras}, the complement of \( Z(F_1) \) in $F$ is nontrivial. That is, there exists \( x \in F_1 \setminus Z(F_1) \), \( x \neq 0 \). Since \( F_1 \subset Z(F) \cup \{0\} \), we have a contradiction. Therefore, the result is established.
\end{proof}

\begin{theorem}
\label{Theonaoabspaceability}For \( p \in (0, \infty] \), if \( F \) is a closed subspace of \( \ell_p \) that contains \( c_0 \), then the set \( F \setminus Z(F) \) cannot be \((\alpha, \mathfrak{c})\)-spaceable for any \( \alpha \geq \aleph_0 \).
\end{theorem}

\begin{proof}
Assume, by way of contradiction, that $F\setminus Z(F)$ is $(\alpha
,\mathfrak{c})$-spaceable for some $\alpha\geq\aleph_{0}$ and let $E$ be an
$\aleph_{0}$-dimensional subspace of $F$ such that $E\setminus\left\{
0\right\}  \subseteq Z(F)$. The existence of such a subspace $E$ is guaranteed
by the Corollary \ref{ZFlineableCorollary}. Fix $m_{1,1}>1$ and consider
linearly independent vectors $(x_{n}^{(1)})_{n=1}^{\infty},\ldots
,(x_{n}^{(m_{1,1})})_{n=1}^{\infty}$ in $E$. Since the $m_{1,1}$ vectors
$(x_{n}^{(1)})_{n=1}^{m_{1,1}-1},\ldots,(x_{n}^{(m_{1,1})})_{n=1}^{m_{1,1}-1}$
are linearly dependent in $\mathbb{K}^{m_{1,1}-1}$, there exists a non-trivial
linear combination of them resulting in the zero vector. Thus, we can assume,
without loss of generality, that there are scalars $a_{2},\ldots,a_{m_{1,1}}$
such that
\[
(x_{n}^{(1)})_{n=1}^{m_{1,1}-1}+\sum_{i=2}^{m_{1,1}}a_{i}(x_{n}^{(i)}%
)_{n=1}^{m_{1,1}-1}=0.
\]
Let
\[
v_{1,1}:=(x_{n}^{(1)})_{n=1}^{\infty}+\sum_{i=2}^{m_{1,1}}a_{i}(x_{n}%
^{(i)})_{n=1}^{\infty}.
\]
It is evident that $v_{1,1}\in W_{m_{1,1}-1}\cap\lbrack E\setminus\{0\}]$.

We proceed as follows:

\begin{enumerate}

\item Consider $r_{1,1}=\min\left\{  t\in\mathbb{N}\cup\{{0\}}:x_{m_{1,1}%
+t}^{(1)}+\sum_{i=2}^{m_{1,1}}a_{i}x_{m_{1,1}+t}^{(i)}\neq0\right\}  $ and
$\varepsilon_{1,1}=1/2$. 

\item Consider also $m_{1,2}=m_{1,1}+r_{1,1}+1$ and fix $z_{1,1}%
=(z_{1,n}^{(1)})_{n=1}^{\infty}$ in $W_{m_{1,2}}\cap\lbrack E\setminus\{0\}]$. 

\item Now, take $r_{1,2}=\min\{n\in\mathbb{N}:z_{1,m_{1,2}+n}^{(1)}\neq0\}$,
fix $u_{1,1}=(u_{1,n}^{(1)})_{n=1}^{\infty}$ in $W_{m_{1,2}+r_{1,2}-1}%
\cap\lbrack E\setminus\{0\}]$ and choose $\lambda_{1,1}$ in $\mathbb{K}%
\setminus\{0\}$ such that
\[
|\lambda_{1,1}|>\max\left\{  \frac{|u_{1,m_{1,2}+r_{1,2}}^{(1)}|}%
{|z_{1,m_{1,2}+r_{1,2}}^{(1)}|},\frac{|v_{1,m_{1,2}+r_{1,2}}^{(1)}%
|+\varepsilon_{1,1}|u_{1,m_{1,2}+r_{1,2}}^{(1)}|}{\varepsilon_{1,1}%
|z_{1,m_{1,2}+r_{1,2}}^{(1)}|}\right\}  .
\]

\item Define $v_{1,2}=\lambda_{1,1}z_{1,1}-u_{1,1}$ $($denote $v_{1,2}$ by
$(v_{1,n}^{(2)})_{n=1}^{\infty})$.
\end{enumerate}

The above construction ensures that $v_{1,m_{1,2} + r_{1,2}}^{(2)} =
\lambda_{1,1} z_{m_{1,2} + r_{1,2}}^{(1)} - u_{1,m_{1,2} + r_{1,2}}^{(1)}
\neq0 $ and $v_{1,n}^{(2)} = 0 $ whenever $n \in\{1, \ldots, m_{1,2} - 1\} $
and that $|v_{1,m_{1,2} + r_{1,2}}^{(1)}| \leq\varepsilon_{1,1} |v_{1,m_{1,2}
+ r_{1,2}}^{(2)}| $.

Proceeding similarly as in the proof of Lemma \ref{Lema das m primeiras},
$($with $\varepsilon_{1,k}=1/2^{k}$ in each step$)$, we obtain a sequence of
vectors $(y_{1,n})_{n=1}^{\infty}$ in $F\setminus Z(F)$ where $y_{1,1}%
=(y_{1,n}^{(1)})_{n=1}^{\infty}$ in particular satisfies:

\begin{itemize}

\item $y_{1,m_{1,1} + r_{1,1}}^{(1)} \neq0 $, 

\item $y_{1,n}^{(1)} = 0 $ for each $n < m_{1,1} $.
\end{itemize}

We claim that
\[
\Vert y_{1,1}-v_{1,1}\Vert_{p}\leq\sum_{k=1}^{\infty}2^{-qk},
\]
where $q=1,$ if $p\in\lbrack1,\infty]$ and $q=p,$ if $p\in(0,1)$. In fact, if
we consider the sequence $(x_{n,1})_{n=1}^{\infty}$ used to obtain the vector
$y_{1,1}$ in the construction above, we obtain
\[
||x_{j,1}-x_{j-1,1}||_{p}\leq\varepsilon_{1,j}^{q}=2^{-qj}\,\,\text{for
each}\,\,j\in\mathbb{N}\text{.}%
\]
This shows that
\begin{align*}
||x_{j,1}-x_{0,1}||_{p}  & \leq||x_{j,1}-x_{j-1,1}||_{p}+||x_{j-1,1}%
-x_{j-2,1}||_{p}\cdots+||x_{1,1}-x_{1,0}||_{p}\\
& \leq\varepsilon_{1,j}^{q}+\varepsilon_{1,j-1}^{q}+\ldots+\varepsilon
_{1,1}^{q}\\
& \leq\sum_{j=k}^{\infty}(\varepsilon_{1,k})^{q}=\sum_{k=1}^{\infty}%
2^{-qk}\text{,}%
\end{align*}
for each $j\in\mathbb{N}$, as desired.

Consequently,
\[
||y_{1,1}-v_{1,1}||_{p}=||\displaystyle\lim_{j\rightarrow\infty}%
x_{j,1}-x_{0,1}||=\displaystyle\lim_{j\rightarrow\infty}||x_{j,1}%
-x_{0,1}||\leq\lim_{j\rightarrow\infty}\sum_{k=1}^{\infty}(\varepsilon
_{1,k})^{q}=\sum_{k=1}^{\infty}2^{-qk}\text{,}%
\]
as desired.

Now, let us construct a vector $y_{2,1}=(y_{2,n}^{1})_{n=1}^{\infty}$ in
$F\setminus Z(F)$ and positive integers $m_{2,1}$ and $r_{2,1}$ such that:

\begin{itemize}

\item $y_{2,m_{2,1} + r_{2,1}}^{(1)} \neq0 $, 

\item $y_{2,n}^{(1)}=0$ for each $n<m_{2,1}$, $($in particular, $y_{2,m_{1,1}%
+r_{1,1}}^{(1)}=0)$, 

\item $\Vert y_{2,1}-v_{1,1}\Vert_{p}\leq\sum_{k=1}^{\infty}3^{-kq}$,
\end{itemize}

To achieve this, start by choosing $m_{2,1}=m_{1,1}+r_{1,1}+1$ and
$r_{2,1}=r_{1,1}$. Define also $v_{2,1}=v_{1,1}$ and denote $v_{2,1}$ by
$(v_{2,n}^{(1)})_{n=1}^{\infty}$. 
We proceed as follows:

\begin{enumerate}

\item Pick $\varepsilon_{2,1}:= 1/3$, consider $m_{2,2}=m_{2,1}%
+r_{2,1}+1$ and fix $z_{2,1}=(z_{2,n}^{(1)})_{n=1}^{\infty}$ in $W_{m_{2,2}%
}\cap\lbrack E\setminus\{0\}]$. 

\item Let $r_{2,2}=\min\{n\geq1:z_{2,m_{2,2}+n}^{(1)}\neq0\}$, fix
$u_{2,1}=(u_{2,n}^{(1)})_{n=1}^{\infty}$ in $W_{m_{2,2}+r_{2,2}-1}\cap\lbrack
E\setminus\{0\}]$ and choose a scalar $\lambda_{2,1}$ in $\mathbb{K}%
\setminus\{0\}$ such that
\[
|\lambda_{2,1}|>\max\left\{  \frac{|u_{2,m_{1,2}+r_{1,2}}^{(1)}|}%
{|z_{2,m_{1,2}+r_{1,2}}^{(1)}|},\frac{|v_{2,m_{2}+r_{2}}^{(1)}|+\varepsilon
_{2,1}|u_{2,m_{1,2}+r_{1,2}}^{(1)}|}{\varepsilon_{2,1}|z_{2,m_{1,2}+r_{1,2}%
}^{(1)}|}\right\}  .
\]

\item Define $v_{2,2}:=\lambda_{2,1}z_{2,1}-u_{2,1}$ $($denote $v_{2,2}$ by
$(v_{2,n}^{(2)})_{n=1}^{\infty})$.
\end{enumerate}

Proceeding similarly as in the proof of Lemma \ref{Lema das m primeiras}, with
$\varepsilon_{2,k}=3^{-k}$ for each $k\in\mathbb{N}$, we obtain a sequence of
vectors $(y_{2,n})_{n=1}^{\infty}$ in $F\setminus Z(F)$ with $y_{2,1}$
satisfying the desired conditions.

Proceeding recursively (starting with $v_{j+1,1}=v_{1,1}$ and $m_{j+1,1}%
=m_{j,1}+r_{j,1}+1$ and taking $\varepsilon_{j+1,k}=\left(  j+2\right)  ^{-k}$
for each $k\in\mathbb{N}$) at the $(j+1)$-th step, we will have in hand the
sequence $\{y_{n,1}\}_{n=1}^{\infty}$ of vectors in $F\setminus Z(F)$, which
satisfies: 

\begin{itemize}

\item $y_{n,m_{n,1} + r_{n,1}}^{(1)} \neq0 $, 

\item $y_{n,m_{i,1} + r_{i,1}}^{(1)} = 0 $ for each $i < n $, 

\item $\Vert y_{n,1}-v_{n,1}\Vert_{p}=\Vert y_{n,1}-v_{1,1}\Vert_{p}%
\leq\displaystyle\sum_{k=1}^{\infty}\left(  n+1\right)  ^{-qk}$.
\end{itemize}

It is plain that the vectors $(y_{n,1})_{n=1}^{\infty}$ are linearly
independent. We claim that $\text{span} \{y_{n,1}\}_{n=1}^{\infty}$ is
contained in $F \setminus Z(F) $. Indeed, let $w$ be a element in $\text{span}
\{y_{n,1}\}_{n=1}^{\infty}$, then there are scalars $a_{1}, \ldots, a_{N}$
such that
\[
w = \sum_{i=1}^{N} a_{i} y_{i,1} = \left(  a_{1} y_{1,1}^{(1)} + a_{2}
y_{2,1}^{(1)} + \ldots+ a_{N} y_{N,1}^{(1)}, \, a_{1} y_{1,2}^{(1)} + a_{2}
y_{2,2}^{(1)} + \ldots+ a_{N} y_{N,2}^{(1)}, \, \ldots\right) .
\]
By Lemma \ref{Lemmadoszeros2}, there exists a increasing sequence of positive
integers $(s_{k})_{k=1}^{\infty}$ such that $y_{n,s_{k}}^{(1)}=0,$ for all $n
\in\{1, \ldots, N\}$ and for all $k \in\mathbb{N}.$ Thus, denoting $w =
(w_{n})_{n=1}^{\infty} $ we have that $w_{s_{k}}=0,$ for all $k \in
\mathbb{N},$ and therefore, $w \in F \setminus Z(F),$ as desired.\newline

Now, since the set $\{y_{n,1}\}_{n=1}^{\infty}$ is linearly independent, we
have
\[
\text{dim}\,\text{span}\{y_{n,1}\}_{n=1}^{\infty}=\aleph_{0}.
\]
On the other hand, since
\[
\underset{j\rightarrow\infty}{\lim}\sum\limits_{k=1}^{\infty}j^{-qk}%
=\underset{j\rightarrow\infty}{\lim}\dfrac{1}{1-j^{q}}=0\text{,}%
\]
we can infer that $v_{1,1}\in Z(F)\cap\overline{\text{span}}\{y_{n,1}%
\}_{n=1}^{\infty}$. If we assume without loss of generality that $\Vert
y_{n,1}\Vert_{p}=1$ for each $n\in\mathbb{N}$, and that $(y_{n,1}%
)_{n=1}^{\infty}$ form an $\ell_{\infty}$-independent sequence (see
Proposition \ref{Proposition sequence l-independent}), the linear operator
\[
T\colon\ell_{1}\rightarrow F,\quad(t_{k})_{k=1}^{\infty}\mapsto\sum
_{k=1}^{\infty}t_{k}2^{-k}y_{k,1}%
\]
is well defined and injective. Thus, if we take an $\alpha$-dimensional
subspace $W_{\alpha}$ such that
\[
\text{span}\{y_{n,1}\}_{n=1}^{\infty}\subseteq W_{\alpha}\subseteq T(\ell
_{1}),
\]
then, since $Z(F)\cap\overline{\text{span}}\{y_{n,1}\}_{n=1}^{\infty}%
\neq\varnothing$, we obtain that $Z(F)\cap\overline{W_{\alpha}}\neq
\varnothing$ which leads to a contradiction. Therefore, the result is done.\\
\end{proof}

The combination of Theorems \ref{TheorempabspaceabilityZX} and \ref{Theonaoabspaceability} provides the information that for \( p > 0 \), if \( F \) is a closed subspace of \( \ell_p \) that contains \( c_0 \), then the set \( F \setminus Z(F) \) is \((\alpha, \mathfrak{c})\)-spaceable if and only if \(\alpha\) is finite. This result highlights the robust structure of the \( \ell_p \) space concerning the concept of \((\alpha, \mathfrak{c})\)-spaceability, when the subspace \( F \) includes \( c_0 \). It emphasizes that any subspace generated by \(\alpha\) elements in \( F \setminus Z(F) \) can be extended to a subspace of maximal dimension still within \( F \setminus Z(F) \), reflecting a significant degree of linearity.

However, it remains an open question whether this result holds when \( F \) is a subspace of \( \ell_{\infty} \setminus c_0 \).

\section*{Acknowledgements}
We would like to thank Professor Dr. Anselmo Raposo Jr. for his valuable contributions during the development of the lemmas in Section 4.


\begin{thebibliography}{99}                                                                                               %
\bibitem{Mikaela} M. Aires, G. Botelho, \textit{Spaceability of sets of non-injective maps}, preprint. arXiv:2403.19855.

\bibitem{Nacib} N.G. Albuquerque, L. Coleta, \textit{Large structures within the class of summing operators},  J. Math. Anal. Appl. \textbf{526}, 2 (2023).

\bibitem {Araujo}G. Ara\'{u}jo, A. Barbosa, A. Raposo Jr., and G. Ribeiro,
\textit{On the spaceability of the set of functions in the Lebesgue space
}$L_{p}$\textit{ which are not in }$L_{q}$, Bull Braz Math Soc, New Series
\textbf{54}, 44 (2023).

\bibitem {Aron}R.M. Aron, L. Bernal-Gonz\'{a}lez, D. Pellegrino and J.B.
Seoane-Sep\'{u}lveda, \textit{Lineability: the search for linearity in
mathematics}, Monographs and Research Notes in Mathematics. CRC Press, Boca
Raton, FL, 2016.

\bibitem {AGSS}R.M. Aron, V.I. Gurariy and J.B. Seoane-Sep\'{u}lveda,
\textit{Lineability and spaceability of sets of functions on }$\mathbb{R}$,
Proc. Amer. Math. Soc. \textbf{133} (2004), 795--803.

\bibitem {Bernal} L. Bernal-Gonz\'{a}lez and M.O. Cabrera, \textit{Lineability
criteria, with applications}, J. Funct. Anal. \textbf{266} (2014), 3997--4025.

\bibitem {S lineability} L. Bernal-Gonz\'{a}lez, J.A. Conejero, M.
Murillo-Arcila, J.B. Seoane-Sep\'{u}lveda, $[S]$\textit{-linear and convex
structures in function families}, Linear Algebra Appl. \textbf{579} (2019), 463--483.

\bibitem{BRSH} L. Bernal-González, D.L. Rodríguez-Vidanes, J.B. Seoane-Sepúlveda, H.J. Tag, \textit{New Results in Analysis of Orlicz-Lorentz spaces}, preprint. arXiv:2312.13903.

\bibitem {Botelho}G. Botelho, V.V. F\'{a}varo, D. Pellegrino and J.B.
Seoane-Sep\'{u}lveda, $L_{p}[0,1]\setminus\cup_{q>p}L_{q}[0,1]$ \textit{is
spaceable for every} $p>0$, Linear Algebra Appl. \textbf{436} (2012), 2963--2965.
\bibitem{Sheldon} M. Caballer, S. Dantas, D.L. Rodríguez-Vidanes, \textit{Searching for linear structures in the failure of the Stone-Weierstrass theorem}, preprint. arXiv:2405.06453. 

\bibitem {Cariello}D. Cariello, J.B. Seoane-Sep\'{u}lveda, \textit{Basic
sequences and spaceability in }$\ell_{p}$\textit{ spaces}. J. Funct. Anal.
\textbf{266} (2014), 3797--381

\bibitem {Diogo/Anselmo}D. Diniz and A. Raposo Jr, \textit{A note on the geometry of
certain classes of linear operators.} Bull. Braz. Math. Soc. (N.S.) \textbf{52} (2021), 1073--1080.

\bibitem {Diogo}D. Diniz, V.V. F\'{a}varo, D. Pellegrino and A. Raposo Jr,
\textit{Spaceability of the sets of surjective and injective operators between
sequence spaces.} Rev. R. Acad. Cienc. Exactas F\'{\i}s. Nat. Ser. A Mat. RACSAM \textbf{114} (2020), no. 4, Paper No. 194, 11 pp.



\bibitem {Drewnowski}L. Drewnowski, I. Labuda and Z. Lipecki,
\textit{Existence of quasi-bases for separable topological linear spaces,}
Arch. Math. (Basel) \textbf{37} (1981), no. 5, 454--456.

\bibitem{Enflo} P.H. Enflo, V.I. Gurariy, J.B. Seoane-Sepúlveda, \textit{Some results and open questions on spaceability in function spaces}, Trans. Amer. Math. Soc. \textbf{366} (2) (2014) 611–625

\bibitem {Pilar}V.V. F\'{a}varo, D. Pellegrino and P. Rueda, \textit{On the size of
the set of unbounded multilinear operators between Banach spaces}, Linear
Algebra Appl. \textbf{606} (2020), 144--158.

\bibitem {FPT}V.V. F\'{a}varo, D. Pellegrino and D. Tom\'{a}z,
\textit{Lineability and spaceability: a new approach}, Bull. Braz. Math. Soc.
New Ser. \textbf{51} (2019), 27--46.

\bibitem {FDAG}V.V. F\'{a}varo, D. Pellegrino, A. Raposo Jr., and G. Ribeiro,
\textit{General criteria for a strong notion of lineability}, Proc. Amer.
Math. Soc. \textbf{152} (2024), 941--954.

\bibitem {Fonf} V. Fonf, V.I. Gurariy, V. Kadec
\textit{An infinite dimensional subspace of \( C[0,1] \) consisting of nowhere differentiable functions}
C. R. Acad. Bulgare Sci., \textbf{52} (1999), pp. 13-16

\bibitem {Gurariy-Quarta} V.I. Gurariy, L. Quarta, \textit{On lineability of sets of continuous functions}, J. Math. Anal. Appl. \textbf{294} (1) (2004) 62–72.

\bibitem {Kalt}N.J. Kalton, \textit{Quasi-Banach spaces}, Handbook of the
Geometry of Banach Spaces, Vol. 2, 1099-1130, North-Holland, Amsterdam, 2003.

\bibitem {Levine e Milman}B. Levine and D. Milman, \textit{On linear sets in
space }$C$\textit{ consisting of functions of bounded variation}, Comm. Inst.
Sci. Math. M\'{e}c. Univ. Kharkoff [Zapiski Inst. Mat. Mech.] (4) \textbf{16}
(1940), 102--105 (Russian, with English summary).

\bibitem {Pellegrino}D. Pellegrino and A. Raposo Jr, \textit{Pointwise
lineability in sequence spaces}. Indag. Math. (N.S.) \textbf{32} (2021), 536--546.

\bibitem{RR} A. Raposo Jr., G. Ribeiro, \textit{Pointwise linear separation property and infinite pointwise dense lineability}, preprint. arXiv:2311.09110.
\bibitem {Stiles}W.J. Stiles, \textit{On properties of subspaces of} $\ell
_{p}$, $0<p<1$, Trans. Amer. Math. Soc. \textbf{149} (1970), 405-415.
\end{thebibliography}
\end{document}